\documentclass{article}
\usepackage{lipsum} 
\usepackage{sidecap}
\usepackage{float}
\usepackage{cite}
\usepackage{tikz}
\usetikzlibrary{arrows,shapes,trees}
\usepackage{lipsum} 
\usepackage[xindy]{glossaries}
\usepackage{multicol}
\usepackage{enumitem}
\usepackage[T1]{fontenc} \usepackage{tgtermes}  

\usepackage{color}
\newenvironment{proof}{\vspace{0.05cm} \nid \emph{Proof.}}{\hfill$\square$ \vspace{0.25cm}}
\newenvironment{Proof}[1]{\vspace{0.25cm} \nid \emph{Proof #1.}}{\hfill$\square$ \vspace{0.25cm}}

\usepackage{enumitem}
\newlist{myQuoteEnumerate}{enumerate}{2}
\setlist[myQuoteEnumerate,1]{label=(\alph*)}
\setlist[myQuoteEnumerate,2]{label=(\alph*)}

\usepackage[a4paper, total={5.5in, 9in}]{geometry} 
\usepackage{amsmath,amssymb,latexsym}
\usepackage[utf8]{inputenc}
\usepackage[english]{babel}
\usepackage{booktabs}
\newtheorem{thm}{Theorem}[section]
\newtheorem{thmm}{Theorem}[section]

\newtheorem{cor}{Corollary}[thm]

\newtheorem{lem}{Lemma}[section]

\newtheorem{defi}{Definition}[section]

\newtheorem{remark}{Remark}[section]

\newcommand{\parr}[1]{\vspace{0.1cm} \noindent {\it #1}}

\usepackage{tocloft}
\usepackage{siunitx}
\usepackage{xcolor}
\usepackage{booktabs,colortbl, array}
\usepackage{pgfplotstable}
\definecolor{rulecolor}{RGB}{0,71,171}
\definecolor{tableheadcolor}{gray}{0.92}

\newcommand{\nid}{\noindent}

\newcommand\quotient[2]{
        \mathchoice
            {
                \text{\raise1ex\hbox{$#1$}\Big/\lower1ex\hbox{$#2$}}%
            }
            {
                #1\,/\,#2
            }
            {
                #1\,/\,#2
            }
            {
                #1\,/\,#2
            }
    }

\makeatletter \renewcommand\paragraph{\@startsection{paragraph}{4}{\z@} {2.25ex \@plus1ex \@minus.2ex} {-0.7em}{\normalfont\normalsize\bfseries}} \makeatother

\definecolor{aurometalsaurus}{rgb}{0.43, 0.5, 0.5}
\definecolor{darkjunglegreen}{rgb}{0.1, 0.14, 0.13}
\definecolor{coolblack}{rgb}{0.0, 0.18, 0.39}
\definecolor{cobalt}{rgb}{0.0, 0.28, 0.67}

\usepackage{hyperref}

\title{Dynamics of hyperbolic correspondences}

\author{CARLOS SIQUEIRA}

\begin{document}

\maketitle




\begin{abstract}
This paper establishes the geometric rigidity of certain holomorphic correspondences in the family $(w-c)^q=z^p,$ whose post-critical set is finite in any bounded domain of $\mathbb{C}.$ In spite of being rigid on the sphere, such correspondences are $J$-stable by means of holomorphic motions when viewed as maps of $\mathbb{C}^2.$ The key idea is the association of a conformal iterated function system to the return branches near the critical point, giving a global description of the post-critical set and proving the hyperbolicity of these correspondences.
\end{abstract}

\tableofcontents

\section{Introduction}

Holomorphic correspondences are relations $z\mapsto w$ given by $p(z,w) = 0$, where  $p$ is a polynomial in two complex variables. 
Finitely generated Kleinian groups and rational maps are two types of correspondences. \emph{Matings} between Kleinian groups and rational maps provide a third type. An important topic of the theory consists in deciding when an abstract mating between a polynomial map and a Kleinian group $\Gamma$ can be realised by a holomorphic correspondence. For example, if $\Gamma$ is Hecke group, matings between pinched polynomial-like maps $f$ and $\Gamma$ are generated by holomorphic correspondences if, and only if, $f$ is Chebyshev-like \cite{Bullett2005}.  Bullett and Penrose \cite{Bullett1994} discovered a family of matings whose connectedness locus is (conjecturally) homeomorphic to the Mandelbrot set. This phenomenon suggests an interesting \emph{dictionary} between the dynamics of certain holomorphic correspondences (e.g.  matings), rational maps  and Kleinian groups \cite{Bullett2001}.   The theory of holomorphic motions \cite{SS17}  has already been extended  to the family $\mathbf{f}_c$ of all multifunctions $z\mapsto w$  determined by 
 \begin{equation}(w-c)^{q} = z^p,\end{equation} where $p>q.$ This is  a \emph{formal} generalisation of the quadratic family. If \label{tvbc} $\beta = p/q$ in reduced form, then $\mathbf{f}_c(z) =z^{\beta} +c,$ where  $z^{\beta} = \exp \beta \log(z)$.   
We shall see in this paper that each $\mathbf{f}_c$ presents features of both polynomial maps and  \emph{conformal iterated function systems (CIFS).}  

{\subsection{Rigidity.} A dynamical system is described as rigid if every deformation either keeps the system unchanged (up to isomorphism), or else radically alters its dynamics, for example by destroying topological conjugacy.

Post-critically finite rational maps are rigid. Critically finite correspondences (those for which the full orbit of every critical point is finite) are rigid; indeed in the quadratic case there are only eleven critically finite correspondences up to conformal conjugacy (see Bullett \cite{Bullett92}).

{Julia  sets of hyperbolic rational maps $f$ are \emph{geometrically rigid} in the sense that if $J(f)$ is a Jordan curve, then either $J(f)$ is a circle in $\hat{\mathbb{C}}$ or $$ \dim_{H} J(f) > 1,$$ where $\dim_H$ stands for Hausdorff dimension. (More recently, Urba\'nski \cite{Urbanski09} has extended this result to a class of meromorphic functions with finitely many singularities).
 In this example,   we cannot deform the circle without transforming it into a fractal. In Theorem \ref{kfmcs} we have another example of geometric rigidity where a finite set cannot be deformed without transforming it into a Cantor set.} 

\subsection{Geometric rigidity for correspondences.} The Julia set $J_c$ can also  be defined for $\mathbf{f}_c$; it is the closure of all repelling cycles (see page \pageref{kfbew} for a detailed reference to all terminology).  The analysis of the post-critical set plays a central role in the dynamics of rational maps, mainly because of the following two properties: $(a)$ the set of attracting cycles is always finite for rational maps $f$ and  $(b)$ every attracting cycle attracts the orbit of a critical point of $f.$ {In general, the closure of the set of attracting cycles of $\mathbf{f}_c$ is uncountably infinite.} Property $(b)$ still holds for the family of holomorphic correspondences $\mathbf{f}_c.$  Therefore, the set of attracting cycles is still very important to study the dynamics of $\mathbf{f}_c$, but its structure can be very complicated, as the following theorem indicates.  For this reason, we define the \emph{dual Julia set} $J_c^*$ as the closure of all  attracting cycles of $\mathbf{f}_c$, except $\infty.$

{The following result shall be restated as Corollary \ref{thbn}.}

\begin{thmm}[Geometric rigidity] \label{kfmcs} {If the critical point of  $\mathbf{f}_a$  has only one bounded forward orbit and if this orbit is necessarily a cycle containing zero,} then the corresponding dual Julia set $J_a^*$ is finite and $J_c^*$ is Cantor set, for every $c$ in a neighbourhood of $a.$

\end{thmm}

 Under the conditions of Theorem \ref{kfmcs}, the postcritical set $P_a$ is finite in any bounded domain of $\mathbb{C}$ and $P_c$ is  accumulated in a Cantor set when $c\neq a$ is sufficiently close to $a.$ Theorefore, Theorem \ref{kfmcs}  is somewhat related to the post-critically finite rigidity of rational maps. As remarked above, critically finite correspondences are rigid. If a correspondence is not critically finite,   there is still the chance that the post-critical set is finite in every bounded domain, i.e., that the only possible accumulation point is $\infty.$  This is precisely the structure of the post-critical set $P_a$  under  the hypothesis of Theorem \ref{kfmcs}.   However, $P_c$ presents uncountably many accumulation points, for any infinitesimal perturbation $c$ of $a.$ This fact could be used to show that $\mathbf{f}_a: \mathbb{C} \to \mathbb{C}$ is not structurally stable {(a precise meaning of \emph{structural stability of multifunctions} involves $\mathbb{C}^2$ extensions -- see \cite{SS17} for a detailed reference). }

\subsection{Conformal iterated function systems.} \label{dfdsd} 
The geometric rigidity established in Theorem \ref{kfmcs} comes from the association of an iterated function system whose limit set is precisely the closure of attracting periodic orbits. We are going to illustrate this idea in a particular example -- see \mbox{Figure \ref{fdt}. } The interplay between iterated function systems and multifunctions is a well known topic in Conformal Dynamics and the situation described in Figure \ref{fdt} is just one instance. See, for example, Bousch \cite{Bousch} M\"unzner \& Rasch \cite{Rasch}, and Przytycki \& Urba\'nski \cite{ConformalF}.

In the following Figure \ref{fdt},  the Julia set of $c_1=-1$ is drawn in black, and some of the bounded components of the Fatou set $\hat{\mathbb{C}} -J_{c_1}$  (which are conformally isomorphic to the unit disk) are coloured in gray, with labels $1, 2, \ldots, 5.$

If we denote the $i$-th component in the figure by $C_i$, then the following sequence of maps of components 
$$C_1 \xrightarrow{\varphi_1} C_2 \xrightarrow{\varphi_2} C_3 \xrightarrow{\varphi_3} C_4 \xrightarrow{\mathbf{f}_c} C_5 \xrightarrow{\varphi_5} C_4 $$
is eventually periodic, where  each  $\varphi_i$ is a conformal isomorphism, for $i=1,2,3,5.$ 

Since the critical point $0$ belongs to $C_4,$ there is no univalent branch defined on $C_4.$  Under  $\mathbf{f}_{c_1}$, every point $z$  in $C_4$  has  two images in $C_5,$   except when $z=0.$  The fifth component is mapped back onto the fourth by a univalent branch $\varphi_5,$ making a cycle of components of length $2.$  We  define the \emph{filled Julia set} $K_c$ as the union of $J_c$ with all bounded components of the Fatou set.  The structure of $K_{c_1}$ can be deduced from Figure \ref{fdt}. 
 \begin{figure}[H]
\centering
 \includegraphics[scale=1.1]{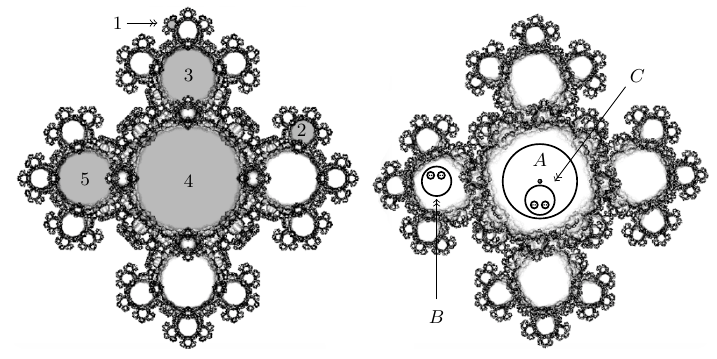}
\caption{Julia sets of $(w+c)^2=z^4$ for $c=-1$ and $c=-1+i/10$, respectively. } 
\label{fdt}
\end{figure}

It turns out that every point in $K_c$ has one bounded forward orbit; on the other hand, every forward orbit of a point in the complement of $K_c$ converges to $\infty$ exponentially fast.   Except for the fifth component (the central one in the figure), every bounded component $U$ of the Fatou set of $\mathbf{f}_{c_1}$  which does not contain the critical point determines two univalent branches of the correspondence; one of them sends $U$ onto another bounded component  contained in $K_c,$ and  the other sends $U$ to the complement of $K_c.$ This gives a global description of the dynamics on the Fatou set.

    For this particular example, every component is pre-periodic and is eventually  mapped into the cycle of components (fourth and fifth disks).   What makes this example especial is that \begin{equation} \label{sdgasd} 0\mapsto -1 \mapsto 0\end{equation} is the only bounded orbit of the critical point.  Under this combinatorial condition, it is possible to prove that every forward orbit in the Fatou set either escapes to infinity or converges to the cycle \eqref{sdgasd}. Since every nonzero point has uncountably many forward orbits, such well organised behaviour is the first indication  of the hyperbolicity of $\mathbf{f}_{c_1}.$

    The right-hand side of this figure displays the Julia set of $\mathbf{f}_{c_2}$, where $c_2$ is a small perturbation of the parameter $c_1$. A Cantor set is produced near the cycle  \eqref{sdgasd}, and every forward orbit of a point in the Fatou set of $\mathbf{f}_{c_2}$ either escapes to infinity or accumulates in this Cantor set. In the following paragraph we shall explain the construction that yields this Cantor set. 
        
    In Figure \ref{fdt}, the correspondence $\mathbf{f}_{c_2}$ maps the disk $A$ around the critical point onto the disk $B,$ which is centred at $c_2.$ Clearly, the mapping $\mathbf{f}_{c_2}$ is not single-valued and every point in $A$ has two images in $B,$ except for the critical point, which is mapped to $c_2.$   In a second step, $\mathbf{f}_{c_2}$ determines two univalent branches on $B;$ one of them sends $B$ into the complement of $K_{c_2},$ and the other branch, say $\varphi$, maps $B$ onto the disk $C.$ The following fact is very important: $C$ is contained in $A$ and also avoids the critical point. Furthermore, the image of $C$ under $\mathbf{f}_{c_2}$ splits into two disks \emph{inside of $B.$} Now $\varphi$ sends the two disks contained in $B$ onto two disks contained in $C.$ The process can be repeated indefinitely, producing two conformal iterated function systems $\mathcal{G}_i$ whose limit sets $\Lambda(\mathcal{G}_i)$  are Cantor sets contained in $B$ and $C$, respectively. 
    
    The general version of the  following result  is Theorem \ref{lqe}.

    \begin{thmm}\label{sdfdc} The limit set $\Lambda$ obtained from the two conformal iterated function systems $\mathcal{G}_i$ above is precisely the closure of all attracting cycles; that is
    $$J_{c_2}^{*} = \Lambda(\mathcal{G}_1) \cup \Lambda(\mathcal{G}_2). $$ Furthermore, the post-critical set $P_{c_2}$ is accumulated in this Cantor set and $\mathbf{f}_{c_2}$ expands the hyperbolic metric of the Riemann surface $\mathbb{C} - P_{c_2}.$ As a consequence, the Julia set is hyperbolic. 
    
    \end{thmm}

    Every correspondence $(w-a)^{2}=z^{2d} $ with  $a^{d-1}=-1$ and $d>1$ displays the same features described in Figure \ref{fdt} and Theorem \ref{sdfdc}.  In all such examples,  $a$ is a simple centre; that is, the critical point has only one bounded orbit and this orbit is a cycle.   In the next section we explain why simple centres are important.

 \subsection{Combinatorial description of the parameter space.} In the quadratic family, the post-critically finite rigidity is related to a special  coding of the hyperbolic components of the interior of the Mandelbrot set $M.$ Every hyperbolic component  $U$ of  the interior of $M$  has a \emph{centre}, corresponding to a post-critically finite $f_c$ determined by the unique $c\in U$ such that $f_c(z)=z^2 +c$ has an attracting cycle with multiplier $\lambda=0.$    See Douady and Hubbard \cite{DH84,DH85}.  Every centre is thus determined by a solution of $f_c^n(0)=0.$  The period $n$ is not fixed, and by collecting all possible solutions of $f_c^n(0)=0$ we obtain a coding of infinitely many components of the interior of $M.$  
 
Although $f_c: \mathbb{C} \to \mathbb{C}$ is  rigid  at the centre  $c\in U,$ all quadratic maps in $U$ are $J$-stable by means of holomorphic motions, see Ma\~n\'e, Sad \& Sullivan \cite{Mane1983}. (Somewhat similar results for $\mathbf{f}_c$ are Theorems \ref{kfmcs} and \ref{qpi}).

 Following the same principle, we say that $c$ is a \emph{centre} for the family $\mathbf{f}_c$  if only finitely many forward orbits of $0$ return to $0$ and the others are discarded to the basin of infinity. Equivalently,  we  say that $c$ is a centre if $\mathbf{g}_c^n(0)=0,$ for some $n>0,$ where $\mathbf{g}_c$ is the multifunction defined on the filled Julia set $K_c,$ given by $\mathbf{g}_c(z) = \mathbf{f}_c(z) \cap K_c.$ (Since   $K_c$ consists of every $z$ with at least one bounded  forward orbit under $\mathbf{f}_c$, we conclude that $\mathbf{g}_c$ is well defined).

 \label{fxmv}

  The study of centres is the easiest approach to the hard problem of classifying topologically the dynamics of all  correspondences in the family. This problem is related to the topological properties of the  \emph{connectedness locus} $M_{\beta}$ of the family $(w-c)^q=z^p,$ where $\beta=p/q$ and $\beta>1.$ By definition, $M_{\beta}$ is the set of all parameters $c$ such that $K_c$ is connected. 
 
 The set  $M_{\beta,0} $ consists of all $c\in \mathbb{C}$ such that $0\in K_c.$ For the quadratic family, this set is equal to $M_{\beta}$. Therefore, the  two sets corresponding to $\beta=2$ are equal to the Mandelbrot set $M.$  Topology and Dynamics are related by the well known  fact: \emph{if the boundary of $M$ is locally connected then hyperbolic maps are dense in the quadratic family. }
 
  In this paper we show that $M_{\beta,0} \subset M_{\beta}.$ Moreover, every centre belongs to $M_{\beta,0}$ and our results concern the dynamical properties of correspondences near centres --  such as Theorem \ref{sopcv} below -- providing an initial picture of  $M_{\beta}$ as a space of dynamical systems. (For more informations on $M_{\beta}$ and a new class of Julia sets named {\it Carpets}, see  \cite{BLS}).

  {The following result is also presented as Theorem \ref{jst}.}
  
  \begin{thmm} \label{sopcv} There is an open set $H_{\beta}$ containing $\mathbb{C} -M_{\beta,0}$ and every simple centre, such that  $c\mapsto J_c$ and $c\mapsto J_c^*$ are continuous on $H_{\beta}$ with respect to  the Hausdorff distance. \end{thmm}   According to Theorem 3.4 of Siqueira \& Smania \cite{SS17}, this theorem implies structural stability and the existence of branched holomorphic motions, {as stated in Theorem \ref{plmd}, which is the general version  of the following result.}

  \begin{thmm}[Holomorphic motions] \label{qpi} If $c_0$ is in $H_{\beta},$ there is a normal branched holomorphic motion $\mathbf{h}_c$ based at $c_0$  and parameterised on $U,$ such that $\mathbf{h}_c(J_{c_0})=J_c,$ for every $c\in U.$ 
  \end{thmm}

 Branched holomorphic motions and structural stability are related by $\mathbb{C}^2$ extensions,   as explained in \cite[Theorem 2.1]{SS17}. 
  Indeed, for every $c$ in the parameterisation domain $U$ of the branched motion $\mathbf{h}$ there corresponds a holomorphic map $f_c: V \to \mathbb{C}^2,$ where $V$ is an open subset of $\mathbb{C}^2.$ The closure of periodic points of $f_c$ is denoted by $J(f_c).$ The dynamics of $\mathbf{f}_c$ on $J_c$ is a factor of $$f_c: J(f_c) \to J(f_c).$$ The sets $J(f_c)$ are related by a (non-branched) holomorphic motion $h:U\times J(f_{c_0}) \to J(f_c)$ whose projection onto $\mathbb{C}$  is precisely $\mathbf{h}.$ 

Since every $h_c: J(f_{c_0}) \to J(f_c)$ is a topological conjugacy, this construction in dimension 2 reveals that $\mathbf{f}_c$ is  structurally stable on $J_c$ when viewed as a map of $\mathbb{C}^2,$  at every parameter  in $H_{\beta}$.

  \subsection{Hyperbolic correspondences.} In this paper we give a precise meaning of hyperbolicity for correspondences.  Recall that (i) \emph{a quadratic map $f$ is hyperbolic if the critical point zero is attracted to an attracting cycle.} This definition presents little information on the dynamics, but it is well known that if  $f$ is hyperbolic then (ii) {\emph{the  union of all the attracting cycles (including the one at infinity) }attracts the whole Fatou set.} The two  properties are equivalent, and  the second  can be extended to correspondences,  yielding some  results which are well known for the quadratic family. There is also a generalisation of hyperbolicity for  finitely generated holomorphic families of rational semigroups $$G_c=\langle f_{1,c}, \cdots , f_{n,c}\rangle$$ satisfying certain conditions. Sumi  has proved that if  $G_a$ is hyperbolic,  then the Julia set $J(G_c)$ moves holomorphically at $c=a$ and $G_c$ is hyperbolic, for every $c$ in a neighbourhood of $a$ (see Sumi \cite{Sumi98}).

In our context, we say that the correspondence $\mathbf{f}_c$ is hyperbolic if $$\hat{\mathbb{C}} - J_c=\left\{ z \in \hat{\mathbb{C}}: \omega(z, \mathbf{f}_c) \subset J_c^* \cup \{\infty\} \right\}, $$ where $\omega(z,\mathbf{f}_c)$ denotes the  $\omega$-limit set of a point  $z.$

{The following result is also stated as Corollary \ref{dgeqd}.}

    \begin{thmm}[Hyperbolicity] \label{pqh} There is an open set $H_{\beta}$ containing $\mathbb{C} - M_{\beta,0}$ and  every simple centre such that $\mathbf{f}_c$ is hyperbolic, for every $c$ in $H_{\beta}.$ Moreover, backward orbits of any point in $\mathbb{C}  - P_c$ accumulate on $J_c$ when $c\in H_{\beta}.$   \end{thmm}

 This result    could be used  to show that  $\mathbf{f}_c$ has no wandering domains when $c\in H_{\beta},$ as Figure \ref{fdt} indicates for the particular example $(w+1)^q=z^p.$  We do not  prove this result in this paper.

    The second assertion  of Theorem \ref{pqh} is the basis of the algorithm used to generate Figure \ref{fdt} and reveals that the notion of Julia set introduced in this paper (closure of repelling periodic orbits) displays an expected feature already known for limit sets of Kleinian groups.

 Recall that for a non-elementary Kleinian group $\Gamma$, there are several equivalent definitions of its limit set $\Lambda(\Gamma)$: it is the complement of its domain of discontinuity $\Omega(\Gamma),$ and also the closure of repelling fixed points; moreover, $\Lambda(\Gamma)$ is the   set of accumulation points of any full orbit under $\Gamma.$

The foundations for a systematic study of regular and limit sets of general holomorphic correspondences are given by Bullett \& Penrose in  \cite{Bullett2001}.
 The generalisations of the various equivalent definitions of  $\Lambda(\Gamma)$ are no longer equivalent in the larger category of holomorphic correspondences.
 It is possible to give a general  definition of \emph{regular domain}  \cite{Bullett1994}, and a notion of \emph{equicontinuity set} for a correspondence  \cite{Bullett2001}.  The Julia set of a correspondence as defined in this paper satisfies  at least two of the equivalent definitions of the Julia set of a rational function: it is the closure of repelling periodic cycles, and for  every $c$ in $H_{\beta}$,  backward orbits of any point in $\mathbb{C}  - P_c$ accumulate on $J_c.$

\section{First steps: recurrence near simple centres}

 We start with some preliminary results on the speed of convergence of iterates in the basin of infinity.

 Recall that a \label{podemg}\emph{multifunction} from $A$ to $B$ is any relation $\mathbf{h} \subset  A\times B$ such that, for every $z$ in $A$, there is $w$ in $B$ such that $(z,w)$ belongs to $\mathbf{h}.$ The notations $\mathbf{h}: A \to B,$   $$ \mathbf{h}(z)=\{w\in B: (z,w) \in \mathbf{h}\}  \ \ \textrm{and}  \ \ \mathbf{h}(A) = \bigcup_{z\in A} \mathbf{h}(z)$$ are commonly used.  \label{pdmge}
 The composition  $\mathbf{h}_1\circ \mathbf{h}_2$  of multifunctions is defined by the relation consisting of all pairs $(z, w)$ for which there is $x$ such that $(z,x)\in \mathbf{h}_2$ and $(x,w)\in \mathbf{h}_1.$  If a multifunction is surjective, that is, $\mathbf{h}(A) = B,$ then its inverse  $\mathbf{h}^{-1}: B\to A$  always exist. In  this case, $\mathbf{h}^{-1}(w)$ is the set of all $z\in A$ such that $(z,w) \in \mathbf{h}.$    \label{phgen}  \label{asene}

 Each member of the one-parameter family of holomorphic correspondences $(w-c)^q=z^p$ (where the integers  $p>q \geq 1$ are fixed) is a multifunction  $\mathbf{f}_c: \hat{\mathbb{C}} \to \hat{\mathbb{C}}.$   Therefore, $\mathbf{f}_c(z)$ \label{pmvbe} is the set of all $w$ such that $(w-c)^q = z^p.$
 
\label{pwerf}A sequence $(z_i)$ is a \emph{forward orbit} (or simply, orbit) if  $z_{i+1} \in \mathbf{f}_c(z_i),$ for every $i.$ If $z_{i+1} \in \mathbf{f}_c^{-1}(z_i)$ for every $i,$ then $(z_i)$ is a \emph{backward orbit}. A set $B$ is called \emph{forward invariant} if $\mathbf{f}_c(B) \subset B.$ If $\mathbf{f}_c^{-1}(B) \subset B,$ then $B$ is \emph{backward invariant. }    If $k>0$, then $$\mathbf{f}_c^{k} = \underbrace{\mathbf{f}_c \circ \cdots \circ \mathbf{f}_c}_{k} \ \ \textrm{and} \ \  \mathbf{f}_c^{-k} = (\mathbf{f}_c^{-1})^{k}. $$
\label{oplf}
 \noindent  {Fix  any $\lambda>1.$} Given  $c\in \mathbb{C},$ the equation  $x^{\frac{p}{q}} -\lambda x -|c| =0$ has  two solutions in $x.$ Let $x_0$ be the greatest solution. Any $R> x_0$  is, by definition, \label{ghbne}\hypertarget{xp}{\emph{an escaping radius of}} $\mathbf{f}_c.$  The following result shows that the region determined by $|z| > R$ is forward invariant under $\mathbf{f}_c$, and that every orbit in this region converges exponentially fast to $\infty.$ 

\begin{lem}[Escaping radius]
\label{rvb} An escaping radius \hypertarget{ls}{$R$} can be chosen locally constant at every $c\in \mathbb{C}.$ If $|z| >R$, then any forward orbit $(z_i)$ of $z$ under $\mathbf{f}_c$ satisfies 
 $$\cdots |z_n| > \lambda |z_{n-1}| > \lambda^2 |z_{n-2}| > \cdots >\lambda^n|z|. $$ Consequently, $(z_i)$ converges \hypertarget{qo}{exponentially} fast to $\infty,$ and\label{gdgihg} $$ \mathbf{f}_c^{-m-1}(B_{R}) \Subset \mathbf{f}_c^{-m}(B_{R}),$$
where  \label{ugnew}$B_R = \{z\in \mathbb{C}: |z| < R\}.$ \end{lem}

\begin{proof} Let $R$ be an escaping radius of $\mathbf{f}_c.$ If $w$ is an image of $z$ under $\mathbf{f}_c$ and $|z| > R,$ then 
$|w|  \geq |w -c| -|c| = |z|^{p/q} -|c| > \lambda |z|.  $
\end{proof}

As a corollary,  the \label{gueng}\emph{filled Julia set}
\begin{equation}\label{urt}
  K_c=\bigcap_{k=0}^{\infty} \mathbf{f}_c^{-k}(B_R)= \bigcap_{k=0}^{\infty} \overline{\mathbf{f}_c^{-k}(B_{R})}\end{equation}

\noindent is nonempty. This definition is independent of $R.$ 
 Due to the following lemma, a point belongs to $K_c$ iff it has at least one bounded forward orbit.  

\begin{lem}[Choice] \label{pfv}  Suppose $z\in A$ and, for every $n>0,$ there is a forward orbit $z\mapsto z_1 \mapsto \cdots \mapsto z_n$ contained in $A.$ Then there is an infinite forward orbit of $z$ in $A.$ (Similarly for backward orbits). 

\end{lem}

\begin{proof} Use the fact that every point has at most $q$ images to find repeated terms and extract an infinite orbit out of infinitely many finite orbits. 
\end{proof}
  \subsection{Connectedness locus.} It follows from \eqref{urt} that $\mathbf{f}_c^{-1}(K_c)=K_c.$ Hence every point in $K_c$ has one infinite forward orbit in $K_c$ and $\mathbf{g}_c: K_c \to K_c$ is the well defined multifunction  given by \label{pohnf}$$\mathbf{g}_c(z) = \mathbf{f}_c(z)\cap K_c,$$  for every $z\in K_c.$ \label{gueng}
 We let \label{qwegn}$M_{\beta}$ be the set of all $c$ such that $K_c$  is connected, and \label{pmngi} $$M_{\beta,0}=\{c\in \mathbb{C}: 0\in K_c\}. $$ We intend to show that  $M_{\beta}$ contains $M_{\beta, 0}.$ For the quadratic family, both sets are equal to the Mandelbrot set and $M_{2} - M_{2,0}$ is empty. (If $\beta$ is sufficiently close to $1$, then the complement of $M_{\beta,0}$ in $M_{\beta}$ represents a large class of filled Julia sets named Carpets (see \cite{BLS} for more details).

 \begin{thm} \label{trv} For any rational $\beta >1,$ the set $M_{\beta, 0}$ is contained in the connectedness locus  $M_{\beta}.$

  \end{thm} 
 
 We shall need the following lemma.   (Recall that a \emph{region} is a nonempty open connected set).
\begin{lem} \label{xwp} If $V=\mathbf{f}_c(U)$ and $0\in U,$ then $U$ is a region iff  \, $V$ is a region. \end{lem}

 \begin{proof} Connectedness in the main concern.  Notice that $U$ is a union of connected sets $L$ containing $0$ such that $\mathbf{f}_c(L)$ is connected (e.g., if $L$ is a line segment, then $\mathbf{f}_c(L)$ consists of $q$ line segments.)  The sets  $\mathbf{f}_c(L)$ have a point in common, namely, $c.$ The union of connected sets with a point in common is also connected.  Hence $V$ is connected if $U$ is connected. \end{proof}

 \begin{Proof}{of Theorem \ref{trv}}  If $0\in K_c,$ then $c$ has a bounded forward orbit and, in particular, $c$ is in $ \mathbf{f}_c^{-m}(B_R),$ for every $m,$ where $R$ is an escaping radius.   It follows from Lemma  \ref{xwp} that every $\mathbf{f}_c^{-m}(B_{R})$ is connected. Hence  the nested intersection $K_c=\bigcap_{m}\overline{\mathbf{f}_c^{-m}(B_{R})}$ is connected. (The nested intersection of compact and connected subsets of a metric space is compact and connected).  This proves that $M_{\beta,0} \subset M_{\beta}.$
 \end{Proof}
 \subsection{Recurrence.}  A forward orbit $(z_i)$ which is also a periodic sequence is called a \emph{cycle}. \label{pgmen}\emph{Centres} are those parameters $c$ for which every forward orbit of $0$ under $\mathbf{f}_c$ either escapes to $\infty$ or comes back to $0$ forming a cycle. More precisely,  a parameter $c$ is a \emph{centre} if  $\mathbf{g}_c^n(0) =\{0\},$ for some $n>0.$ In particular, if there is only one bounded forward orbit of $0$ under $\mathbf{f}_c$ and this orbit is necessarily a cycle,   then $c$ is by definition a \hypertarget{st}{\emph{simple centre.}}  (See the Introduction for examples of simple centres).

A holomorphic map $\varphi$ is a \label{pjyfk}\emph{branch} of $\mathbf{f}_c$ if $(\varphi(z) -c)^q=z^p,$ for every $z$ in the domain of $\varphi.$

A cycle $(z_i)_0^n$ of period $n$ is said to be \emph{attracting} in two situations: (i) when zero belongs to the cycle (in this case, it is called \emph{super-attracting)}; or (ii) when  the \emph{multiplier} $$\lambda= \prod_{i=0}^{n-1} \varphi_i'(z_i)  $$ satisfies $|\lambda| <1,$ where each $\varphi_i$ is a branch of $\mathbf{f}_c$ taking $z_i$ to $z_{i+1}$ ($\varphi_i$ exists only if $z_i$ is nonzero).  If $|\lambda | >1,$ the cycle is said to be \emph{repelling}.

  Let  $B_d=\{|z| <d\}.$  Every region which is conformally isomorphic to $B_d$ is  a \label{qejtn}\hypertarget{cp}{\emph{conformal disk}}, or simply, a disk. Suppose $c=a$ is a simple centre. Then there is only one bounded forward orbit $(z_i)$ of zero under $\mathbf{f}_a$, and this orbit is necessarily a cycle of length $n.$ We say that  $n$ is a minimal period if no other element of the cycle is zero, except for the first one.

  \begin{thm} \label{pcvm}Suppose $a$ is a simple centre, with a cycle $(z_i)$ of minimal period $n.$ For every $d$ in some interval \hypertarget{interval}{$(0, \epsilon)$} and $c$ in a neighbourhood $V$ of $a$, there is only one sequence of \hypertarget{cm}{maps}  \begin{equation} \label{hdf}
B_d \xrightarrow{\mathbf{f}_c} D_{1,c} \xrightarrow{\varphi_{1,c}} D_{2,c} \to \cdots \xrightarrow{\varphi_{n-1,c}} D_{n,c}, \ \ \ c\in V,
 \end{equation}
 
\noindent satisfying: \hypertarget{s1}{$(a)$} every $D_{i,c}$ is a disk  and $\varphi_{i,c}: D_{i,c} \to D_{i+1,c}$ is a conformal isomorphism; \hypertarget{s2}{$(b)$} for every $c$ fixed, the disks $D_{i,c}$  are pairwise disjoint and $D_{n,c} \Subset  B_d;$  \hypertarget{s3}{$(c)$} every $\varphi_{i,c}$ is a branch of $\mathbf{f}_c,$ and\label{olq}
 \label{tghm} there are branches $\phi_i$ of $\mathbf{f}_0$ such that $\varphi_{i,c}= \phi_i +c$ on $D_{i,c},$ for every $c\in V$ (in this sense, we say that $\varphi_{i,c}$ is a \hypertarget{perturbation}{\emph{perturbation}}  of $\varphi_{i,0})$; and $(d)$ the orbit of zero under  successive compositions of the maps $\varphi_{i,a}$ is precisely the cycle $(z_i).$

\end{thm}

\begin{proof} A sequence of disks satisfying $D_{n,a}\Subset B_d$ exists because the cycle $(z_i)$ contains zero and is super-attracting. The maps $\varphi_{i,a}$ are uniquely determined. By continuity, a small perturbation of each $\varphi_{i,a}$ yields a family of  sequences \eqref{hdf} with $D_{n,c} \Subset B_d,$ for every $c$ in a neighbourhood of $a.$
\end{proof}

The \emph{dual Julia set} $J_c^*$ is the closure of all attracting cycles, except $\infty.$  For rational maps, this set is always finite, but for correspondences its geometry can be as puzzling as the Julia set \label{qetfx}$J_c,$ \label{qegyc} which is by definition, the closure of  all \emph{repelling cycles} of $\mathbf{f}_c.$

We shall prove that $J_c^*$ is a Cantor set associated with a conformal iterated function system when $c$ is sufficiently close to a simple centre.

The first step towards this assertion is  Theorem \ref{pcvm}. The second step is provided by the following result, Theorem \ref{ppp}. It can be regarded as  an improvement of the previous theorem, in the sense that $D_{n,c}$ is now proved to avoid the critical point.

Since $\mathbf{f}_c(z) = c+  \exp  \frac{1}{q}\log z^{p},$ for every disk $D$ avoiding $0$ there are $q$   univalent branches $\varphi_j: D \to \mathbb{C}$ determining $\mathbf{f}_c(D)=\bigcup_j \varphi_j(D).$ If this union is pairwise disjoint, $D$ is said to be a \label{pohbg}\emph{univalent disk}. The polar coordinates form of $\mathbf{f}_c$  is $(r, \theta) \mapsto (r^{p/q}, \frac{p}{q} \theta)+ c.$ Thus $\mathbf{f}_c$ bijectively transforms sectors of angle $2 \pi/ p$ at $0$ to sectors of angle $2\pi/q$ at $c.$ This fact can be used to prove the following lemma.  \begin{lem}\label{plmt} Any disk $D$ contained in an open sector of angle $2\pi/ p$ and vertex $0$ is a univalent disk. 
\end{lem}

The family of sequences of Theorem \ref{pcvm} depends on two parameters $d$ and $c.$  For this reason we shall denote it by \label{poihyt} $\mathcal{S}_{c,d}$ in the following theorem.

\begin{thm}[Recurrence] \label{ppp}   For every $c$ in a punctured neighbourhood $V-\{a\}$ of a simple centre $a,$ there is $d(c)$ such that the sequence $\mathcal{S}_{c,d}$ satisfies:
\label{ghj}
(a) the radius $d(c)$ of the first disk $B_d$ \hypertarget{lb}{approaches} zero as $c\to a;$ (b)
 $D_{n,c} \Subset B_d-\{0\}$ is a univalent disk;  (c) some iterate $\mathbf{f}_c^{-k}(B_{R})$ is a connected set containing the union of all disks of the sequence  $\mathcal{S}_{c,d}$;  and (d)    every $D_{i,c}$ is mapped into $q$ different disks under $\mathbf{f}_c$, and only one of them, $D_{i+1,c}$, is contained in $\mathbf{f}_c^{-k}(B_{R}).$

\end{thm}

We shall include here an analysis of Figure \ref{fdt} of the Introduction to illustrate this theorem. In that particular example, $n=2$ and $B_d$ is the disk labeled as $A$ in Figure \ref{fdt}. By Theorem \ref{ppp} $(a)$,  the size of the disk $A$ is very small when $c$ is close to the simple centre $-1.$ As a consequence, the Cantor set $J_c^*$ described in Figure \ref{fdt} is very close to the cycle $ 0 \mapsto -1 \mapsto 0$ when $c$ is sufficiently close to  $-1.$

The statement  $(b)$ of Theorem \ref{ppp} simply means that the disk $C=D_{2,c}$ avoids $0$ and its image splits into $2$ small disks contained in $B= D_{1,c}.$ The region $\mathbf{f}_{c}^{-k}(B_R)$ contains the filled Julia set, and can be visualised in Figure \ref{fdt} as a good approximation of $K_c$ for high values of $k.$ 

\begin{Proof}{of Theorem \ref{ppp}}  For $c\in V,$ let $g_c: D_{1,c} \to D_{n,c}$ be given by $g_c(z)=\varphi_{n-1,c} \circ \cdots \circ \varphi_{1,c}(z).$ 
If  $(z_{i,c})_{i=0}^{n}$ denotes  the orbit of $z_{0,c}=0$  under the sequence of maps in \eqref{hdf}, then $z_{c,n} =g_c(c).$  Let $f: V\to \mathbb{C}$ be given by $f(c)=g_c(c).$ Since $\varphi_{i,c}(z) = \phi_i(z) +c$ for a certain branch $\phi_i$  of $z\mapsto z^{\beta}$ with $\beta=p/q,$ it is possible  to show that $f(z)$ is indeed holomorphic and non-constant. Hence, $f: V \to \mathbb{C}$ is an open map.  (For example, if  $n=4$ then $f(z)$ is a branch of the multifunction
$((z^{\beta} +z)^{\beta} +z)^{\beta} +z,$ and therefore $f(z)$ cannot be constant).

There is a constant $C_0$ such that  $|g_c'(z)| \leq C_0,$ for every $c\in V$ and $z\in D_{1,c}.$ Since $D_{1,c}= \mathbf{f}_c(B_d)$, it follows that  $D_{1,c}$ is the open ball $B (c, d^{\beta}).$  Notice that $g_c$ maps the centre of $D_{1,c}$ to $f(c).$ Using the mean value inequality applied to $g_c: D_{1,c} \to D_{n,c},$ we conclude that  $D_{n,c}$ is contained in the open ball of radius $C_0d^\beta$ and centre $f(c).$ 

Main argument: in order to show $(a)$ and $(b),$ we shall prove that there is a neighbourhood $V_1$ of $a$ such that, for every $c$ in $V_1-\{a\}$, there is  $d$ satisfying (i)   $|f(c)| + C_0d^{\beta} < d$ and  (ii)  $C_0 d^{\beta} < |f(c)| \sin(\pi/p).$  Since $D_{n,c}$ is contained in the open ball of radius $C_0d^\beta$ and centre $f(c),$ the first inequality implies $D_{n,c} \Subset B_d.$ Using basic trigonometry we conclude that (ii) implies  
$D_{n,c} \Subset S_{2\pi/p},$ for some open sector $S_{2\pi/p}$ of angle $2\pi/p$ and vertex $0.$ According to Lemma \ref{plmt}, the disk $D_{n,c}$ is univalent and compactly contained in $B_{d} - \{0\}.$

Recall that $\beta=p/q$ and let $\gamma= \beta -1.$ There is  $d_0>0$ with $3C_0d_0^{\gamma} < \sin (\pi/p).$ Let $r=C_0 d_0^{\beta}/ \sin (\pi/p).$ Since $a$ is an isolated zero of $f$, the set $f^{-1}(B_r)$ contains an open set $V_1 \subset V$ around $a$ such that   $0<|f(c)| < r,$ for every $c\in V_1-\{a\}.$   If we choose  $d$ with $y:=C_0d^{\beta}/\sin (\pi/p)$ sufficiently close to $|f(c)|,$ then   $y < |f(c)| < 2y.$ 
Since $y<|f(c)| < r,$ it follows that $d< d_0.$ By the assumption,  $3C_0d_0^{\gamma} < \sin (\pi/p).$ Since $d<d_0,$ we have   $3C_0 d^{\beta} < d\sin (\pi/p).$ Therefore, $$\frac{C_0 d^{\beta}}{\sin (\pi/p)} < |f(c)|<\frac{2C_0 d^{\beta}}{\sin (\pi/p)} < d - \frac{C_0 d^{\beta}}{\sin (\pi/p)} < d - C_0 d^{\beta}, $$  which completes the main argument.

 We now show that we may reduce $V_1$ even more so that, for any $c$ in $V_1-\{a\}$ and $d=d(c),$  the set $\mathbf{f}_c^{-k}(B_{R})$ satisfies assertions $(c)$ and $(d).$ 
 
Since  the cycle $(z_{i,a})_{i=0}^{n}$ is the only orbit of zero under $\mathbf{g}_a,$    $z_{i+1,a}$ is the only image of $z_{i,a}$ under $\mathbf{f}_a$ which is contained in $K_a.$

Hence $\mathbf{f}_a(z_{i,a}) - \{z_{i+1,a}\}$ is contained in the complement of $K_a.$  By \eqref{urt}, there is $m>0$ such that $\mathbf{f}_a^m$ maps $\mathbf{f}_a(z_{i,a}) - \{z_{i+1,a}\}$ into the complement of $B_R,$ for every $0<i<n.$ By continuity, if $c$ is sufficiently close to $a$, then $\mathbf{f}_c^{m}$ maps $$\mathbf{f}_c(D_{i,c}) - D_{i+1,c} $$ into the complement of $B_R$ (notice that the diameter of $D_{i,c}$ tends to zero as $c\to a$). Therefore, $\mathbf{f}_c$ maps $D_{i,c}$ into $q$ different disks, and only one of them, $D_{i+1,c},$ is contained in $\mathbf{f}_c^{-m}(B_R).$
The set $ \mathbf{f}_c^{-m}(B_R)$ is connected (by Lemma \ref{xwp}) and satisfies assertions $(c)$ and $(d).$    \end{Proof}

 Property $(b)$ of Theorem  \ref{ppp} states that $D_{n,c}$ avoids the critical point when $d=d(c).$ The following result  reveals that $D_{n,c}$ contains the critical point in many situations where $d\neq d(c).$

 \begin{thm}  \label{fed}
Let $a$ be a simple centre. For every $d$ in a small interval $(0, \epsilon),$ there is a neighbourhood $V_d$ of $a$ such that, if  $c\in  V_d,$ then the sequence $\mathcal{S}_{c,d}$  of Theorem \ref{pcvm} satisfies
$$0 \in D_{n,c} \Subset B_d,$$ and the {third and fourth properties of Theorem \ref{ppp} still hold in this case. }

 \end{thm} 
 
 \begin{proof} Fix $c=a$ and find an interval $(0,\epsilon)$ such that $0\in D_{n,a} \Subset B_d,$ for every $d$ in $(0, \epsilon).$ Since $D_{n,c}$ is open and varies continuously with $c,$ for every $d$ in $(0, \epsilon)$ there is a neighbourhood $V_d$ such that $0\in D_{n,c} \Subset B_d,$ whenever $c\in V_d.$  \end{proof}

\begin{defi}[Critical  and non-critical systems] \normalfont Any sequence of maps satisfying the conditions of Theorem \ref{ppp} is a \label{sdgrw}\emph{non-critical system} and is denoted by $\mathcal{F}_{c}.$ 
Any sequence satisfying the conditions of Theorem \ref{fed} is a \label{poyhn}\emph{critical system} and is denoted by $\mathcal{A}_c.$

As described in Theorem \ref{pcvm}, in both cases the sequence is given by disks $D_{i,c}.$ We define \label{pitgk}
$$\operatorname{dom}(\mathcal{F}_c)=\bigcup_{i=1}^n D_{i,c}   \ \ \textrm{and} \ \ \operatorname{dom}{(\mathcal{A}_c)} = \bigcup_{i=0}^{n-1} D_{i,c}. $$

\end{defi}

\noindent The main difference between $\operatorname{dom}(\mathcal{F}_c)$ and $\operatorname{dom}(\mathcal{A}_c)$ is that the critical point belongs only to the second.  However, both sets are contained in  $K_c,$ and iterations \emph{inside} them will always {converge to the Cantor set $J_c^*$,} as we shall see in Theorems \ref{lqe}, \ref{ste} and equation \eqref{grw}. 

Iterating inside $\operatorname{dom}(\mathcal{F}_c)$ naturally  produces a conformal iterated function system (CIFS).

\section{Conformal iterated function systems}  \label{spl}

In view of Theorem \ref{ppp}, for every $c\neq a$ sufficiently close to a simple centre $a$ there corresponds a non-critical system $\mathcal{F}_c$ given by \eqref{hdf}.  Since $D_{n,c}$ is  simply connected and $0\not \in D_{n,c},$ there are $q$ holomorphic maps $\psi_{k}: D_{n,c} \to \mathbb{C}$ such that $\mathbf{f}_c(D_{n,c}) = \bigcup_{k} \psi_k(D_{n,c}).$   Let
\begin{equation} \label{qxt}
f_k(z) = \varphi_{n-1,c} \circ \cdots \circ \varphi_{1,c} \circ \psi_{k}(z),
\end{equation}
where the sequence of maps $\varphi_{i,c}$ is determined by $\mathcal{F}_c.$

Since $f_{k}(D_{n,c})\Subset D_{n,c},$ for every $0<k\leq q,$ by the Schwarz-Pick Lemma each $f_{k}$ is a contraction of the hyperbolic metric, and $\{f_k\}_1^q$   defines a CIFS  without overlaps on $D_{n,c}.$  Since $D_{n,c}$ is univalent (see Theorem \ref{ppp} $(b)$), the limit set  \begin{equation} \label{gjcnd}\Lambda_0 = \bigcap_{j>0} H^j(D_{n,c})\end{equation} is a Cantor set,  where $H$ is the Hutchinson operator given by $H(A) = \bigcup_{k=1}^{q}f_k(A),$ for  $A\subset D_{n,c}.$

  By Theorem \ref{ppp} $(c)$-$(d)$, the region $\mathbf{f}_c^{-k}(B_{R})$ contains  $ \operatorname{dom}(\mathcal{F}_c),$ and  $$\mathbf{g}_c(\operatorname{dom}(\mathcal{F}_c)) \subset  \operatorname{dom}(\mathcal{F}_c).$$ Therefore, it makes sense to define \begin{equation} \label{ess}
\Lambda(\mathcal{F}_c)= \bigcup_{j=0}^{n-1} \mathbf{g}_c^{j}(\Lambda_0)
\end{equation}
as the \emph{limit set} of $\mathcal{F}_c.$ In  \eqref{grw} and \eqref{rcv} we present a dynamical interpretation of $\Lambda(\mathcal{F}_c).$

By definition,\label{fhvm} a point $z_*$ belongs to \label{pogew}$\omega(z,\mathbf{g}_c)$ iff there is a forward  orbit under $\mathbf{g}_c$ starting at $z$ which has a subsequence converging to $z_*.$
If we allow only orbits under compositions of maps of $\mathcal{F}_c,$ then we obtain $\omega_{\mathcal{F}_c}(z).$ Notice that every orbit of a point outside $\mathbf{f}_c^{-k}(B_R)$ is attracted to infinity.  Combining this fact with Theorem \ref{ppp} (d) we conclude that, if $z\in \operatorname{dom}(\mathcal{F}_c)$ or $z\in B_d$,  then any forward orbit of $z$ under  $\mathbf{g}_c$ is also an orbit under the maps of $\mathcal{F}_c,$ and \begin{equation} \label{grw}\omega(z,\mathbf{g}_c)= \omega_{\mathcal{F}_c}(z)=\Lambda(\mathcal{F}_c). \end{equation}
The system $\mathcal{F}_c$ induces a CIFS not only on $D_{n,c},$ but also on every $D_{i,c}.$ If $\Lambda_{i}$ denotes the limit set of the CIFS induced on $D_{i,c},$ then \begin{equation} \label{rcv}\Lambda(\mathcal{F}_c)= \bigcup_{i=0}^{n-1} \Lambda_i.\end{equation}

\subsection{Hyperbolic sets.} We say that $\Lambda$ is a \label{poikht}
\emph{hyperbolic attractor} of $\mathbf{g}_c$ if  $\mathbf{g}_c(\Lambda) = \Lambda$ and $\mathbf{g}_c$ attracts a conformal metric  $\| \cdot\|$ defined on a neighbourhood of $\Lambda;$ that is, $\sup \|\varphi'(z)\| <1,$ where $\sup$ is taken over all $z\in \Lambda$ and every branch $\varphi$ of $\mathbf{g}_c$ such that $\varphi(z) \in \Lambda.$   (If we replace $\mathbf{g}_c$ by $\mathbf{f}_c$ in this definition, we obtain a \emph{hyperbolic attractor} of $\mathbf{f}_c.$)

\label{plkv}  Recall that $J_c^*$ was defined as the closure of all attracting cycles. Under the presence of a non-critical system $\mathcal{F}_c$, we shall prove that $J_c^*$ is a hyperbolic attractor obtained from the limit set $\Lambda(\mathcal{F}_c).$ In the following theorem, we say that a cycle $(z_i)_0^n$ attracts a forward orbit $(w_i)_0^{\infty}$ of the critical point $w_0=0$ if $\{z_i\}_ 0^{n}$ is the set of limits of convergent subsequences of $(w_i)_0^{\infty}.$

\begin{thm} \label{lqe} Any attracting cycle of $\mathbf{f}_c$ attracts a forward $\mathbf{f}_c$-orbit of the critical point. If $\mathcal{F}_c$ is a non-critical system,
then the Cantor set $\Lambda(\mathcal{F}_c)$ is  a hyperbolic attractor of $\mathbf{g}_c$ and 
\begin{equation}\label{net} J_c^*=\Lambda(\mathcal{F}_c). 
\end{equation}
\end{thm}

We shall use the following lemma. 

 \begin{lem} \label{solf}If $U$ is simply connected and $c \not \in U,$ then for any $w_0\in U$ and $z_0$ in $ \mathbf{f}_c^{-1}(w_0)$ there is a unique univalent branch $\phi: U \to \mathbb{C}$ of $\mathbf{f}_c^{-1}$ such that $\phi(w_0)=z_0.$
 
 \end{lem}

 \begin{proof}
 A consequence of analytic continuation applied to the algebraic curve determined by $\mathbf{f}_c^{-1}.$ 
 \end{proof}

\begin{Proof}{of Theorem \ref{lqe}} Let $\Lambda_0$ be the limit set defined by \eqref{gjcnd}. Then $\Lambda_0 = h(\Sigma_q),$ where $\Sigma_q$ denotes  $\{1, \ldots, q\}^{\mathbb{N}_0}$ and $$h(\underline{k})= \bigcap_{j} f_{k_{0}} \circ f_{k_{1}} \circ \cdots \circ f_{k_j} (D_{n,c}), $$
for any sequence $\underline{k}=(k_j)$ in $\Sigma_{q}.$  Give $\Sigma_q$ the product topology.

 Periodic points of the left shift $\sigma$ on $\Sigma_q$ are dense in $\Sigma_q,$ and  $h$ maps a periodic  $\underline{k}$ into a periodic $z\in \Lambda_0,$ since $h(\sigma \underline{k}) = \bigcup_j f_{k_1}\circ \cdots \circ f_{k_j}(D_{n,c}).$

\noindent  If $\sigma^j(\underline{k})= \underline{k},$  then $g(z)=z,$ where $g=f_{k_0} \circ \cdots \circ f_{k_{j-1}}.$ Therefore, $z$ is an attracting periodic point of $\mathbf{f}_c.$ Since $h: \Sigma_q \to \Lambda_0$ is continuous and surjective,  $\Lambda_0$ is contained in $J_c^*. $ Using the same reasoning, the limit set $\Lambda_i$ in \eqref{rcv} is also contained in $J_c^*.$ Hence, $\Lambda(\mathcal{F}_c) \subset J_c^*. $

Now we are going to show that {attracting cycles attract a forward orbit of the critical point.} Let $z_k \to \cdots  \to z_1 \to z_0$ be an attracting cycle of $\mathbf{f}_c$ of period $k$, that is, $z_k=z_0.$   Suppose this cycle does not attract any orbit of the critical point, and that no element of this cycle is $0.$  
 
Any point $z$ in a sufficiently small neighbourhood  $B_0$ of $z_0$ is attracted to the attracting cycle.  By a repeated application of Lemma \ref{solf}, we construct a sequence of conformal isomorphisms
 $$\cdots \xrightarrow{\phi_{j+1}} B_j \xrightarrow{\phi_j} B_{j-1} \xrightarrow{\phi_{j-1}} \cdots B_1 \xrightarrow{\phi_1} B_0  $$
where each $B_j$ is a topological disk  containing $z_j$ and the maps $\phi_j$ satisfy the following extension property:  $B_{(k+1)n+j} \supset B_{kn +j},$ and $\phi_{(k+1)n +j}$ restricted to $B_{kn+j}$ is $\phi_{kn+j}.$  This construction is possible because at each step $j,$ we have $c\not \in B_j$ (otherwise the cycle would attract an orbit of the critical point). Hence the family of disks $B_j$ avoid $\{0, c, \infty\}$ (we may suppose $c\neq 0,$ for if $c=0,$  then every point outside the unit circle is either attracted to $0$ or $\infty.$ In this case there is nothing to prove, since  $\mathbf{f}_0$ has  only one attracting cycle, namely, the fixed point $z=0$).  

It follows that $B_{*}= \bigcup_{j=0}^{\infty} B_{jn}$ is a nested union of disks; hence $B_*$ is also a simply connected set avoiding a three point set $\{0,c, \infty\}.$ Therefore, $B_{*}$ is isomorphic to the unit disk. There is a conformal isomorphism $\phi: B_* \to B_*$  such that $$\phi|_{B_{jk}} = \phi_{(j-1)k+1}\circ \cdots \circ \phi_{jk},$$ for every $j.$ In particular, $\phi$ has an attracting fixed point at $z_0.$

By the Schwarz Lemma, the map $\phi: B_* \to B_*$ is conformally conjugate to a rotation of the unit disk, and the existence of an attracting fixed point yields a contradiction. Hence the cycle must attract at least one orbit of the critical point. In other words, every attracting cycle is contained in $\omega(0, \mathbf{g}_c).$  Since $\omega(0, \mathbf{g}_c)$ is closed, $J_c^* \subset \omega(0,\mathbf{g}_c).$ A final application of \eqref{grw} yields $$J_c^*\subset \omega(0, \mathbf{g}_c) = \Lambda(\mathcal{F}_c) \subset J_c^*,$$ thereby proving \eqref{net}. 

It is clear from $\mathbf{g}_c(\operatorname{dom}(\mathcal{F}_c)) \subset \operatorname{dom}(\mathcal{F}_c)$ that $\mathbf{g}_c$ attracts a conformal metric defined on a neighbourhood of $J_c^* = \Lambda(\mathcal{F}_c).$ Therefore, $J_c^*$ is a hyperbolic attractor of $\mathbf{g}_c.$  
\end{Proof}

The  combination of Theorems \ref{lqe}  and  \ref{ppp} establishes  the geometric rigidity of dual Julia sets at simple centres by proving that simple centres are isolated in the parameter space:

\begin{cor}[Geometric rigidity] \label{thbn}The dual Julia set $J_c^*$ is finite at every simple centre $c=a,$ and a Cantor set for any $c\neq a$ sufficiently close to $a.$ 
\end{cor}
There is an analogous of Theorem \ref{lqe}  for critical systems $\mathcal{A}_c$: 
\begin{thm} \label{rrr}   Let $\mathcal{A}_c$ be a critical system given by \eqref{hdf}, i.e., a sequence  $B_d \to \cdots  \to D_{n,c}.$  Then $\operatorname{dom}(\mathcal{A}_c)$  is forward invariant under $\mathbf{g}_c$ and $\mathbf{g}_c^n(B_d) \Subset B_d.$ If some iterate $\mathbf{g}_c^{kn}(B_d)$ is a hyperbolic Riemann surface  contained in a  disk $D \Subset B_{d} - \{0\},$  then the restriction of $\mathbf{g}_c^{n}$ to  $\mathbf{g}_c^{kn}(B_d)$ is given by finitely many conformal maps $g_j: \mathbf{g}_c^{kn}(B_d) \to  \mathbf{g}_c^{kn}(B_d)$ which are contractions with respect to the hyperbolic metric of  $\mathbf{g}_c^{kn}(B_d).$ Therefore, $\{g_j\}$ is a CIFS. If $\Lambda_0$ denotes its limit set, then 
 \begin{equation}
\omega(z,\mathbf{g}_c)= \bigcup_{i=0}^{n-1} \mathbf{g}_c^{i}(\Lambda_0)= J_c^*,
\end{equation}
for  every $z$ in $\operatorname{dom}(\mathcal{A}_c).$

\end{thm}

\begin{proof} A repetition of the arguments used to prove \eqref{qxt},    \eqref{ess}, \eqref{grw} and Theorem \ref{lqe}. \end{proof}

The following is a version of Theorem \ref{ppp} for critical systems. 
\begin{thm}\label{ste} Let $a$ be a simple centre.  For every $d$ in some interval $(0, \epsilon)$ there is a neighbourhood $V_d$ of $a$ such that, for any $c$ in $V_d,$ the sequence $\mathcal{S}_{c,d}$  of Theorem \ref{pcvm}  is a critical system $\mathcal{A}_c$ and \begin{equation}
\label{qpf}
\omega(z, \mathbf{g}_c) = J_c^*, 
\end{equation}
for every $z\in \operatorname{dom}(\mathcal{A}_c).$ 
\end{thm}

\begin{proof} The proof is divided into two steps.

 \noindent {\it Step 1: preliminaries.} By Theorem \ref{fed},  for every $d$ in some interval $(0, \epsilon_1),$ there is a neighbourhood $W_{d}$  of the simple centre $a$ such that the sequence $\mathcal{S}_{c,d}$  in \eqref{hdf} is a critical system, for every $c\in W_{d}.$ 

 For $c\in W_{d}$ and $z\in D_{1,c}$, let  $g_c(z) = \varphi_{n-1,c} \circ \cdots \circ \varphi_{1,c}(z),$ where $(\varphi_{i,c})_{i}$ is the sequence of maps of $\mathcal{S}_{c,d}.$ Therefore, the domain  of $g_c$ is  $D_{1,c}$ and $g_c$ maps  $D_{1,c}$ onto  $D_{n,c}.$  Recall that $D_{1,c} = \mathbf{f}_c(B_d).$
 The following properties hold for the family $g_c:$ 
 
 $(a)$ there exists a constant $C_0 >1 $ such that $$|g_c'(z)| \leq C_0, $$ for every $d\in (0, \epsilon_1),$  $c\in W_d$ and $z\in \mathbf{f}_c(B_d);$ 
 
  $(b)$  $f(c) =g_c(c)$ defines an open holomorphic map $f: W_d \to \mathbb{C}$ with an isolated zero at $c=a.$

 Notice that $f(c)$ is the final point of the orbit of zero under the sequence of maps in $\mathcal{S}_{c,d}.$  Now fix $d_1$ in $(0, \epsilon_1),$ $\lambda = 1/3,$ and $0 < \epsilon < d_1$ such that $C_0\epsilon^{\beta -1} < \lambda$ where $\beta = p/q >1.$ Let $$ V_d = W_d \cap W_{d_1},$$ for every $d \in (0, \epsilon_1).$

 \noindent {\it Step 2: the $V_d$ just defined satisfies the properties of the theorem.}   Indeed, let $d\in (0, \epsilon).$ Then $a\in V_d$ and $\omega(z, g_c) = J_c^*$ if $c=a.$ (As a matter of fact, every orbit under $\mathbf{g}_a$ inside $\operatorname{dom}(\mathcal{S}_{a,d})$ is attracted to the  only cycle $(z_i)_0^n$ of $\mathbf{g}_a$ containing zero. In this case, $J_a^*$ is precisely $\{z_1, \ldots, z_{n-1}\}.$)

We are going to show that $\omega(z, \mathbf{g}_c) = J_c^*$ if $c\in V_d$ and $c\neq a.$    
   Let $\delta = |f(c)|.$ Since $a$ is an isolated zero, $\delta >0.$ If $z\in B_{d},$ then $|z| < \epsilon$ and $|z|^{\beta-1} < \lambda.$ Since $\mathbf{g}_c^n = g_c \circ\mathbf{f}_c$ on $B_d,$ from $(a)$ of Step 1 we conclude that
 \begin{equation} \label{wvb}
 |\mathbf{g}^n_c(z) - f(c)| = |\mathbf{g}_c^n(z) - \mathbf{g}_c^n(0)| \leq C_0 |z-0|^{\beta} < \lambda|z|,
 \end{equation}
 where $|\cdot|$ is also used to denote the diameter of a set. Since $\mathbf{g}_c^n$ is a multifunction, $\mathbf{g}_c^n(z)$ is a set and \eqref{wvb} means that  if we replace $\mathbf{g}_c(z)$ by any $w$ in  $\mathbf{g}_c(z),$ then \eqref{wvb} holds.

 It follows that $\mathbf{g}^n_c$ maps $B_{d}$ into a subset of the ball $B(f(c), r)$ of radius $r=\lambda d.$ Hence $
\mathbf{g}_c^{n}(B_{d})$ is contained in $B_{r_1},$ where $r_1= \delta + \lambda d.$ Using induction, we conclude that $\mathbf{g}_c^{kn}(B_{d})$ is contained in the ball $B_{r_k}$ where $$r_k = \delta + \lambda \delta + \cdots + \lambda^{k-1} \delta + \lambda^{k}d. $$
Choose $k$ such that $\lambda^{k+1}d < \delta/2.$ It follows that $\mathbf{g}_c^{(k+1)n}(B_{d})$ is contained in the ball $B(f(c), s),$ where $$s=\lambda  r_{k} = \lambda \delta + \lambda^2 \delta + \cdots + \lambda^{k} \delta + \lambda^{k+1} d < \frac{\lambda}{1 -\lambda}\delta + \lambda^{k+1}d < \delta.$$ Combining this with Lemma \ref{xwp} we conclude that  $\mathbf{g}_c^{(k+1)n}(B_{d})$ is a connected open set contained in a disk $D\Subset B_{d} - \{0\}.$ A final application of     Theorem  \ref{rrr} yields the desired result. \end{proof}

The following corollary gives an approximation of $J_c^*$ in terms of $\mathbf{g}_c^n(B_d).$

\begin{cor}\label{cwq} If $\mathcal{A}_{c}$ is a critical system in the conditions of Theorem \ref{ste} and $B_d$ denotes the disk of $\mathcal{A}_c$ containing zero, then for every open set $U$ containing $J_c^*$ there is  $n_0>0 $ such that  $$J_c^* \subset \bigcup_{n>n_0} \mathbf{g}_c^{n}(B_d) \subset U.$$
\end{cor}

\begin{proof} Use Theorem \ref{ste}, Theorem \ref{rrr} and the following general fact: if a CIFS is represented by finitely many maps $f_k: X \to X,$ then any high iterate $H^n(X)$ is an approximation of the limit set, where $H(A)=\bigcup_{k} f_k(A)$ is the Hutchinson operator. 
\end{proof}

\section{Post-critical set} The \emph{post-critical set} of \label{qetdbc} $\mathbf{f}_c$ is $ P_c = \overline{\bigcup_{n>0} \mathbf{f}_c^n(0) }.$ Since  $\mathbf{f}_c(P_c) \subset P_c,$ the complement $\mathbb{C}- P_c$ is backward invariant under $\mathbf{f}_c.$

The set of \label{hfvmwe}\emph{accumulation  points} of $P_c$ is denoted by $P_c'.$ Therefore $z\in P_c'$ iff there is a sequence $z_i \neq z$  in $P_c$ converging to $z.$ 

\begin{thm}[Postcritical sets] \label{dstg} Let $a$ be a simple centre of the family  $\mathbf{f}_c.$ 
\begin{itemize} \item[$(a)$]  Suppose $c=a.$ If $P_a \subset K_a,$ then $P_a'=\emptyset;$ otherwise, $P_a'=\{\infty\}.$ \item[$(b)$]  Suppose $c\neq a$ is sufficiently close to $a.$ Then $P_c'$ is a Cantor set.  If $P_c \subset K_c,$ then  $P_c'=J_c^*; $ otherwise,  $P_c' \supsetneq J_c^*  \cup \{\infty\}.$ 
\item[$(c)$] If $c$ is in the complement of $M_{\beta,0}$,  then $P_c'=\{\infty\}.$ 
\end{itemize}
\label{fhwp}
\end{thm}

\begin{proof} $(a)$  If  $P_a \subset K_a,$ then $P_a$ is bounded. Hence every  forward orbit of zero is trapped into the critical system and, consequently, $P_a$ consists of the finitely many points of an  attracting cycle containing zero. Thus,  $P_a'= \emptyset.$ 

If  $P_a$ is not contained in $K_a,$ then some orbit of zero escapes to infinity. Therefore, $\infty$ belongs to $P_a'.$ Since $a$  is a simple centre, there is a region  $\mathbf{f}_a^{-k}(B_R)$  satisfying the third and fourth properties of Theorem \ref{ppp}. As a consequence of Theorem \ref{fed}, the intersection of $P_a$ with a certain iterate 
 $\mathbf{f}_a^{-k}(B_R)$ is a cycle containing zero; and since every point of $P_a$ which is not in $\mathbf{f}_a^{-k}(B_R)$ is attracted to infinity exponentially fast, there can be no other accumulation point except $\infty.$  Thus, $P_a'= \{\infty\}$.

$(b)$ Suppose $c\neq a$ is sufficiently close to $a.$ If $\mathbf{f}_c^{-k}(B_R)$ is a region satisfying the third and fourth properties of Theorem \ref{ppp}, then every forward orbit of zero is either trapped into the non-critical system $\mathcal{F}_c$ or eventually leaves $\mathbf{f}_c^{-k}(B_R)$. Hence, $$P_c'\cap \mathbf{f}_c^{-k}(B_R) = \omega(0, \mathbf{g}_c) = J_c^*, $$ where the second inequality comes from Theorem  \ref{lqe} and equation \eqref{grw}. By Corollary \ref{thbn}, $J_c^*$ is a Cantor set. Hence, $P_c'\cap \mathbf{f}_c^{-k}(B_R)$ is a Cantor set. If $P_c \subset K_c,$ then in particular $P_c$ is contained in $\mathbf{f}_c^{-k}(B_R),$ and $P_c'$ is therefore a Cantor set.

Now suppose  $P_c$ is not contained in $K_c. $
Then $$ W= P_c'\cap (\mathbb{C} - \mathbf{f}_c^{-k}(B_R))$$ is nonempty. Indeed, $\overline{\varphi(D\cap J_c^*)} $ is a Cantor set contained in $W,$ for every univalent branch $\varphi: D \to \mathbb{C}$ defined in an open set $D$ intersecting $J_c^*.$  Iterations of $\overline{\varphi(D\cap J_c^*)} $ are still contained in $W$ and converge to $\infty$ in the Hausdorff distance. It can be shown that the whole set $W$ can be described in this way. Hence both $W$ and $P_c'$ are Cantor sets containing $\infty.$

$(c)$  If $c$ is in the complement of $M_{\beta,0},$ then $0$ is not in $K_c.$ Hence  every orbit of zero escapes to $\infty.$ It follows that  $P_c'=\{\infty\}.$ 
\end{proof}

\begin{defi} \label{sdboidg}{We shall use  $H_{\beta}$ to denote the union  of the complement of $M_{\beta, 0}$ with a sufficiently   small neighbourhood of the set of  all simple centres $a\in M_{\beta, 0}.$ We shall see in Lemma \ref{lkc} that $H_{\beta}$ is open.}
\end{defi}


\begin{remark}{\normalfont {If  $H_{\beta}$ appears in two different statements, then we have two different sets $H_{\beta};$ their intersection is again another $H_{\beta}$ which works for both statements.}

{In Corollary \ref{dgeqd} we show that every parameter in $H_{\beta}$ is hyperbolic, for some $H_{\beta}.$ }}
\end{remark}

\begin{thm}[Hyperbolic metric] \label{xtl} For every nonzero parameter $c$ in $H_{\beta},$   the set $\mathbb{C} - P_c$ is a hyperbolic Riemann surface and $\mathbf{f}_c$ expands its  hyperbolic metric. Equivalently,  
\begin{equation}\label{qkl} d_c(\varphi(z), \varphi(w)) > d_c(z,w),  \end{equation}
for any $z,w \in U$ and any  forward  holomorphic branch $\varphi: U \to \mathbb{C} - P_c,$  where \label{lkghegd}$d_c$ is the distance function associated with the hyperbolic metric of  \ $\mathbb{C} - P_c.$ 

For every family $\mathcal{B}_{K}$ of branches $\varphi$ as above such that  $\varphi(U)$ is contained in the same compact set $K\subset \mathbb{C} - P_c, $ there is a constant $\rho >1$ such that every $\varphi 
\in \mathcal{B}_K$ expands $d_c$ by the uniform factor $\rho.$

\end{thm}

This theorem does not hold for  $c=0$ because  $\mathbb{C}-P_0=\mathbb{C}^*$ is not hyperbolic.

\begin{proof}  A key role is played by the Schwarz-Pick Lemma.  We have  to show that  if $c\not \in M_{\beta, 0}$ or $c$ is sufficiently close to a simple centre, then $(a)$ $\mathbb{C} - P_c$ is connected and  $P_c$ contains at least two points;    $(b)$ $\mathbf{f}_c^{-1}(P_c) \supsetneq P_c;$ and $(c)$ every holomorphic branch  $\varphi$ whose image is contained in $\mathbb{C} - P_c$ is indeed an expansion of the hyperbolic metric.

$(a)$ If $c=a$ is a simple centre, then $P_a$ contains $\{0, a\},$ and $a\neq 0.$ By Theorem \ref{dstg}, the set of accumulation points $P'_a$ is either empty or $\{\infty\}.$ Hence $\mathbb{C} -P_a$ is connected and its complement in $\hat{\mathbb{C}}$ contains at least three points. We conclude that $\mathbb{C} -P_a$ is a hyperbolic Riemann surface.  

If $c\neq a$ is close to a simple centre, then $P'_c$ is  a Cantor set, and therefore $\mathbb{C} - P_c$ is connected. (All Cantor sets are homeomorphic. Furthermore, a Cantor set never disconnects the plane, as a consequence of Schoenflies Theorem, which states that any homeomorphism between plane Cantor sets can be extended to a homeomorphism of  $\mathbb{C}.$)

By Theorem \ref{dstg}, if $c\not \in M_{\beta,0},$  then $P_c'$ is a single point set and $P_c$ contains an infinite forward orbit of zero converging to $\infty.$ It follows that  $\mathbb{C}- P_c$ is a hyperbolic Riemann surface. 

\vspace{0.2cm}

\noindent $(b)$ Our second claim is  $\mathbf{f}_c^{-1}(P_c) \supsetneq P_c.$ 

{Case 1: suppose $c\neq a$ is sufficiently close to a simple centre $a.$}  Then $0=\mathbf{f}_c^{-1}(c)$ is a point of $\mathbf{f}_c^{-1}(P_c),$ which has a bounded orbit contained in $\operatorname{dom}(\mathcal{F}_c). $ In particular, $0\in K_c.$ Since $P_c\cap K_c \subset \operatorname{dom}(\mathcal{F}_c)$ and $\operatorname{dom}(\mathcal{F}_c)$ avoids $0,$ we conclude that  $0$ is not in $P_c.$ Hence  $\mathbf{f}_c^{-1}(P_c) \supsetneq P_c$ in this case. 

{Case 2: now $c=a$ is a simple centre.}  Suppose for a moment that $\mathbf{f}_a^{-1}(P_a) =P_a.$  Recall that $p\geq 2$ and $q\geq 1.$ Every $z\neq a$ has at least two pre-images and $C=P_a \cap K_a$ is a cycle containing $0.$  Since both $P_a$ and $K_a$ are backward invariant,  $\mathbf{f}_a^{-1}(C)=C.$  Since $a\neq 0,$ the cycle $C$ must contain at least two points; thus $C=\{a\}\cup C_1$, where $\operatorname{card}(C_1)=N \geq 1$ and every element of $C_1$ has $p\geq 2$  pre-images in $C.$ Different points $x,y$ in $C$ determine disjoint sets $\mathbf{f}_a^{-1}(x)$ and $\mathbf{f}_a^{-1}(y)$ contained in $C.$  Proving that $\mathbf{f}_a^{-1}(x)$ and $\mathbf{f}_a^{-1}(y)$ are disjoint is not immediate because $\mathbf{f}_a$ is a multifunction. In this specific case, since $\mathbf{f}_a^{-1}(x)$ and $\mathbf{f}_a^{-1}(y)$ are contained in $C$, if the two sets are not disjoint, then there is a point $z$ of the cycle $C$ with two different images in $C.$ This implies that $0$ has two different bounded orbits under $\mathbf{f}_a$, but this is impossible because $a$ is a simple centre.

Therefore $$ 1+N=\operatorname{card}(C) =\operatorname{card}\mathbf{f}_a^{-1}(C) \geq 1+2N,$$ which is a contradiction. Therefore, $\mathbf{f}_a^{-1}(P_a) \supsetneq P_a.$

{Case 3: suppose $c\not \in M_{\beta,0}.$}  Equivalently, $0$ is not in $ K_c.$
Since $0$ belongs to $\mathbf{f}_c^{-1}(P_c)$, it suffices to show that $0$ is not in $P_c.$  Since $P_c'=\{\infty\},$  if $0\in P_c$ then $0 \in \mathbf{f}_c^{i}(0),$ for some $i>0.$ Hence $0\in P_c$ implies the existence of a  cycle containing $0,$ which contradicts the fact $0 \not \in K_c.$ 

\vspace{0.2cm}
$(c)$ Suppose $c \in H_{\beta}.$ Consider the Riemann surface $$X=\left\{(z,w) \in \mathbb{C}^2: w \not \in P_c\right\}.$$ The two projections $\pi_1(z,w)=z$ and $\pi_2(z,w)=w$ define holomorphic maps on $X.$ Since $\pi_2|_X$ omits at least 2 points of $\mathbb{C},$ it follows that $X$ is hyperbolic. Since $\pi_2: X \to \mathbb{C} - P_c$ is a covering map, it is an isometry (by the Schwarz Lemma). On the other hand, $\pi_1: X \to \mathbb{C} - P_c$ is not onto because $\mathbf{f}_c^{-1}(P_c) \supsetneq P_c.$ Hence $\pi_1$ is a strict contraction with respect to the hyperbolic metrics of $X$ and $\mathbb{C} - P_c.$  If  $\varphi: U \to \mathbb{C} - P_c$ is a holomorphic branch, then $U \subset \mathbb{C} -P_c$ and $\varphi$ is also a branch of $\pi_2 \circ \pi_1^{-1}$; therefore $\varphi$ expands the hyperbolic metric of $\mathbb{C} -P_c.$
\end{proof}

\section{Stability and hyperbolicity}

A subset $\Lambda$ of the plane is a \label{eughen}
\emph{hyperbolic repeller} for $\mathbf{f}_c$ if  $\mathbf{f}_c^{-1}(\Lambda) = \Lambda$ and $\mathbf{f}_c$ expands a conformal metric  $\| \cdot\|$ defined on a neighbourhood of $\Lambda;$ that is, $\sup \|\varphi'(z)\| >1,$ where $\sup$ is taken over all $z\in \Lambda$ and every branch $\varphi$ of $\mathbf{f}_c$ such that $\varphi(z) \in \Lambda.$

The main goal of this section to prove that, if $c\in H_{\beta},$ then $J_c$  is a hyperbolic repeller which is stable by means of holomorphic motions.   In spite of the several complications of iterations of multivalued functions (e.g., uncountably many orbits of a  point)  the dynamics of $\mathbf{f}_c$ is quite well understood when $c\in H_{\beta}.$ The hyperbolic geometry of $\mathbb{C} -P_c$ is responsible for such a nice behaviour.  

  We define  \label{wcmg} $\alpha(z, \mathbf{f}_c)$ as the set of all $\zeta \in \mathbb{C}$ such that  $\zeta$ is a limit point of a backward orbit of $z$ under $\mathbf{f}_c.$   The following result provides approximations of $J_c,$ and reveals that $\alpha(z, \mathbf{f}_c)$ is independent of $z.$

\begin{thm}\label{lqp} For every $c$ in  $H_{\beta},$ the Julia set $J_c$ is a hyperbolic repeller for $\mathbf{f}_c,$ and $J_c^*$ is a hyperbolic attractor for $\mathbf{g}_c.$ If $z$ is in $ \mathbb{C}- P_c,$ then  \begin{equation} \label{plq}
\alpha(z, \mathbf{f}_c) = J_c.
\end{equation}
In particular, backward orbits are dense in $J_c,$ and  for any $\epsilon$-neighbourhood $U$ of  $J_c$ there is  $n_0>0$ such that  $\mathbf{f}_c^{-n}(z) \subset U,$ for every $n > n_0.$ 
\end{thm}

The proof of Theorem \ref{lqp} will be given in sequence of lemmas.

\begin{lem}\label{wqt} If $c$ is sufficiently close to a simple centre, then $J_c^* \subset \operatorname{dom}(\mathcal{A}_c)$ and $J_c$ is contained in the complement of $\operatorname{dom}(\mathcal{A}_c),$ for any critical system $\mathcal{A}_c.$ In particular, $J_c$ is disjoint from $J_c^*.$

\end{lem}

\begin{proof} The first assertion follows from Corollary \ref{cwq}, since $\mathbf{g}_c^n(B_d)$ is always contained in  $\operatorname{dom}(\mathcal{A}_c).$  For the second, notice that  if  $\operatorname{dom}(\mathcal{A}_c)$ contains one point of a repelling cycle, then the whole cycle must be contained in $\operatorname{dom}(\mathcal{A}_c)$. But any cycle contained in $\operatorname{dom}(\mathcal{A}_c)$ is attracting. \end{proof}

\begin{lem} \label{lkc} If $c$ is in the complement of $M_{\beta,0}$, then $J_c^*$ is empty. Furthermore,  for every neighbourhood $U$ of $\infty$ there is $n_{0}>0$ such that $\mathbf{f}_c^n(0) \subset U,$ for every $n> n_0.$ {Consequently,  $M_{\beta, 0}$ is closed and $H_{\beta}$ is open. }

\end{lem}

\begin{proof}  The dual Julia set is empty because any attracting cycle attracts a bounded orbit of the critical point (Theorem \ref{lqe}), and in this case, no forward orbit of zero is bounded. 

Let $U_1 \subset U$ be a forward invariant neighbourhood of $\infty.$ We are going to show that $\mathbf{f}_c^{n_0}(0) \subset U_1,$ for some $n_0.$ (Consequently, $\mathbf{f}_c^n(0) \subset U,$ for every $n \geq n_0.$). If no iterate $\mathbf{f}_c^{n_0} (0)$ is contained in $U_1,$  than  by  Lemma \ref{pfv} there is a forward orbit of zero contained in $\mathbb{C}-U_1,$ contradicting the fact that every orbit of zero converges to $\infty$ when $c$ is not  in $ M_{\beta,0}.$

{If $c_0$ is not in $M_{\beta,0}$, then there exist $n$ and a neighbourhood of infinity $U$ such that $\mathbf{f}_{c_0}^{n}(0) \subset U;$  by continuity, $\mathbf{f}_{c}^n(0) \subset U$ for every $c$ in a neighbourhood $V$ of $c_0.$ Hence, $0\not \in K_c$ for every $c\in V.$ Therefore, $c_0$ in an interior point of the complement of $M_{\beta,0}.$ Since $c_0$ is arbitrary, $M_{\beta,0}$  is closed and $H_{\beta}$ is consequently open. }
\end{proof}

\begin{lem}\label{vnm} For any parameter $c$ in $H_{\beta}$ there is a region $\mathcal{D}_c$ such that $$\mathbf{f}_c^{-1}(\mathcal{D}_c) \subset \mathcal{D}_c  \ \ \textrm{and} \ \ 
 J_c \subset \mathcal{D}_c  \Subset \mathbb{C} - P_c. $$

\nid Moreover, for any bounded set $A\Subset \mathbb{C} - P_c$ there is $n>0$ such that  $\mathbf{f}_c^{-n}(A) \subset \mathcal{D}_c.$ \end{lem}

\begin{proof} Since $J_c$ is the closure of repelling cycles (which are contained in $K_c$) and $K_c$ is closed, it follows that $J_c \subset K_c.$ Therefore, $J_c\subset \mathbf{f}_c^{-m}(B_R),$ for every $m>0.$

First suppose $c$ is sufficiently close to a simple centre $a.$ By Theorem \ref{fed} there is a critical system $\mathcal{A}_c$ and a region $\mathbf{f}_c^{-k}(B_R)$ containing $\operatorname{dom}(\mathcal{A}_c)$ satisfying the third and fourth properties of Theorem \ref{ppp}. Since $\operatorname{dom}(\mathcal{A}_c)$ is a disjoint union of disks, $\mathcal{D}_{c}= \mathbf{f}_c^{-n}(B_{R}) - \operatorname{dom}(\mathcal{A}_c)$ is connected. From  Lemma \ref{wqt}  we conclude that $\mathcal{D}_c$ contains $J_c.$ Since $c$ is sufficiently close to a simple centre,   $P_c\cap \mathbf{f}_c^{-n}(B_{R}) \Subset \operatorname{dom}(\mathcal{A}_c).$ Taking complements, we obtain:  $\mathcal{D}_c \Subset \mathbb{C} - P_c.$

Let $A\Subset \mathbb{C} -P_c $ be a bounded set. It remains to show that  some backward iterate of $A$ is contained in $\mathcal{D}_c.$ Since  $A$ is bounded, there is $r>0$ such that $A\subset B_r. $ From Lemma \ref{rvb},  there is $k_0$ such that $\mathbf{f}_c^{-k_0}(B_r) \Subset B_R. $   If we let $k_1= k_0 +n,$ and define $A_1$ to be $\mathbf{f}_c^{-k_1}(A),$  then   $A_1\Subset \mathbf{f}_c^{-n}(B_{R}).$
Denote the  Euclidean distance $d(A_1, J_c^*)$ by $\delta.$  Since $\mathbb{C} - P_c$ is backward invariant,  $A_1$ is compactly contained in $\mathbb{C} -P_c.$  By Theorem \ref{lqe}, we  have $J_c^* \subset P_c.$ Thus,  $\delta >0.$ Let $U$ denote  $\{z: d(z,J_c^*)< \delta\}.$ 

Let $E_0 = A_1 \cap \operatorname{dom}(\mathcal{A}_c).$ If $E_0 = \emptyset,$ then $A_1 \subset \mathcal{D}_c$ and there is nothing to prove. Otherwise, let $$E_1= \mathbf{f}_c^{-1}(A_1) \cap \operatorname{dom}(\mathcal{A}_c) = \mathbf{f}_c^{-1}(E_0) \cap \operatorname{dom}(\mathcal{A}_c).$$

\noindent (In  the second equality,  we have used the definition of $\mathcal{A}_c$ and the following facts: $A_1 \subset \mathbf{f}_c^{-n}(B_{R}),$ and $\operatorname{dom}(\mathcal{A}_c)$ is forward invariant under $\mathbf{g}_c$).  Inductively, we obtain a sequence  $E_k\subset \operatorname{dom}(\mathcal{A}_c)$ such that  $$E_k=\mathbf{f}_c^{-1}(E_{k-1}) \cap \operatorname{dom}(\mathcal{A}_c)=\mathbf{f}_c^{-k}(A_1) \cap \operatorname{dom}(\mathcal{A}_c). $$  By Corollary \ref{cwq}, there is $j_0$ such that,  if  $z$  is any point having  a backward orbit of length greater than $j_0$ contained in $\operatorname{dom}(\mathcal{A}_c),$ then $z$ is in $U.$ As it turns out, if $E_k$ is nonempty, then there is an orbit $z_k \to \cdots \to z_1 \to z$ in $\operatorname{dom}(\mathcal{A}_c)$ with $z\in A_1.$ If $k>j_0,$ then $z\in U;$ and since $U$ is disjoint from $A_1,$ it follows that $E_k=\emptyset$ if $k> j_0.$ Conclusion: $\mathbf{f}_c^{-k}(A_1) \subset \mathcal{D}_c, $ and a backward iterate of $A$ is contained in $\mathcal{D}_c$ when $c$ is sufficiently close to  a simple centre.  

 Second case:  $c\not \in M_{\beta,0}.$ By Lemma \ref{lkc}, some backward iterate  $\mathbf{f}_c^{-n}(B_R)$ is compactly contained in $\mathbb{C} - P_c.$ We may take $\mathcal{D}_c = \mathbf{f}_c^{-n}(B_R)$ in this case. If $A\subset \mathbb{C} - P_c$ is bounded, then   $A$ is contained in some ball $ B_r$, for which   there is $m$ such that $\mathbf{f}_c^{-m}(B_r) \Subset B_R.$ Therefore, $\mathbf{f}_c^{-(n+m)}(A) \Subset \mathcal{D}_c.$ 
 \end{proof}

\begin{Proof}{of Theorem \ref{lqp}}  Suppose $c\in H_{\beta}$ and  let $z\in \mathbb{C} - P_c.$ 

{First claim:} {\it   $\alpha(z,\mathbf{f}_c)$ is contained in $J_c.$} If $z_*$ is in $\alpha(z, \mathbf{f}_c)$ and $(z_n)$ is a backward orbit of $z$ with a subsequence converging to   $z_*$, then in view of  Lemma  \ref{vnm}, there is $n_1$ such that $z_n \in \mathcal{D}_c$ for $n\geq n_1.$ 
Consider the distance function $d_c$ given by the hyperbolic metric of $\mathbb{C} - P_c$ (Theorem \ref{xtl}). By Lemma \ref{vnm},  an $\epsilon$-neighbourhood of $\mathcal{D}_c$ is compactly contained in $\mathbb{C} - P_c.$  By Theorem \ref{xtl}, there is $\mu  <1$ such that every branch $\varphi: B(w,\epsilon) \to \mathbb{C}$ of $\mathbf{f}_c^{-1}$ with $w\in \mathcal{D}_c$ contracts distances by $\mu.$ Since $z_{n_j} \to z_*,$ we may suppose   $z_{n_1}$ and $z_{n_2}$ are contained in $B(z_*, \epsilon/4)$ and $\mu^{n_2 -n_1} < 1/4.$ Since $B(z_{n_1}, \epsilon) \subset \mathbb{C} -P_c,$  there is a backward orbit $$\cdots U_{j} \xrightarrow{\varphi_j} \cdots U_{n_1+1}\xrightarrow{\varphi_{n_1 +1}} U_{n_1}= B(z_{n_1}, \epsilon) $$
where every $\varphi_j$ is a conformal isomorphism and $z_j \in U_j.$ It follows that $U_j$ is contained in  $ B(z_j, \mu^{j-n_1} \epsilon),$ and so $U_{n_2}$ is contained in $B(z_{n_2}, \epsilon/4) \subset B(z_{n_1}, 3\epsilon/4).$ Hence $U_{n_2} \Subset U_{n_1},$ and by the Banach fixed point Theorem, there is a fixed point of $\varphi_{n_1+1}\circ \cdots \circ   \varphi_{n_2} $ in $U_{n_2},$  which necessarily comes from a repelling cycle of $\mathbf{f}_c.$  Since $U_{n_2} \subset B(z_{n_2}, \epsilon/4) \subset B(z_*, \epsilon/2)$ and $\epsilon$ is arbitrary, we conclude that $z_* \in J_c.$

{Second claim:} {\it $\alpha(z,\mathbf{f}_c) \subset \mathcal{D}_c$ is closed and independent of $z$ in $\mathbb{C} -P_c.$}   Since  $P_c'$  is either a single point set or a Cantor set (see Theorem \ref{dstg}), for every two points $z$ and $w$ in $\mathbb{C} - P_c$ there is  a simply connected set $U \subset \mathbb{C} - P_c$ such that $z,w\in U.$ Given an infinite backward orbit $(z_n)$ of $z,$ a repeated application of Lemma  \ref{solf} yields a  sequence of univalent branches of $\mathbf{f}_c,$ $$\cdots \to U_n \xrightarrow{\varphi_{n}} U_{n-1} \cdots \xrightarrow{\varphi_1} U_0=U,$$ where each $\varphi_i: U_{i} \to U_{i-1}$ is a conformal isomorphism taking $z_i$ to $z_{i-1}.$  
This sequence of maps also produces a backward orbit $(w_n)$ of $w$   with $\varphi_i(w_i) = w_{i-1}.$
By Lemma \ref{vnm} and Theorem \ref{xtl}, there are $n_0$ and $\mu<1$ such that $U_n \subset \mathcal{D}_c$ and $$\operatorname{diam}(U_n) \leq \mu^{n-n_0} \operatorname{diam}(U_{n_0}),  $$  for every $n>n_0.$  Since $z_n$ and $w_n$ belong to $U_n,$ and the diameter of $U_n$ goes to zero as $n\to \infty,$ every limit point of $(z_n)$ is a limit point of $(w_n).$ We  conclude that $\alpha(z, \mathbf{f}_c) \subset \alpha(w, \mathbf{f}_c).$ The same argument may be used to prove the other inclusion. Hence $\alpha(w,\mathbf{f}_c)= \alpha(z, \mathbf{f}_c).$

In order to prove that $\alpha(z, \mathbf{f}_c)$ is closed, let $x$ be the limit of a convergent sequence $(x_{n})$ in  $\alpha(z,\mathbf{f}_c).$ Let $\epsilon_n$ be a sequence of positive real numbers converging to zero.  Since $x_1$ is in $\alpha(z, \mathbf{f}_c),$ there is a backward orbit $$z_{n_1} \to z_{n_1 -1} \to \cdots \to z_1 \to z$$ such that $z_{n_1}$ is in $ \mathbb{C} - P_c$ and $d(z_{n_1}, x_1) < \epsilon_1.$ Since $x_2$ is in $\alpha(z_{n_1}, \mathbf{f}_c) = \alpha(z, \mathbf{f}_c),$ there is a backward orbit $z_{n_2} \to \cdots \to z_{n_1}$ such that $d(z_{n_2}, x_2) < \epsilon_2.$ Inductively, we construct a backward orbit $(z_n)$  of $z$ such that $d(z_{n_j}, x_j) < \epsilon_j,$ from which we conclude that $x$ is in $\alpha(z, \mathbf{f}_c).$ In other words, $\alpha(z, \mathbf{f}_c)$ is closed.

{Third claim:} {\it the set $\alpha(z, \mathbf{f}_c)$ is backward invariant, i.e., $\mathbf{f}_c^{-1}(\alpha(z, \mathbf{f}_c)) = \alpha(z, \mathbf{f}_c).$} \\ 
Let $x$ be in $\alpha(z, \mathbf{f}_c)$ and $\epsilon_n \to 0.$  By Lemma \ref{vnm},  $x$ is in  $\mathcal{D}_c\subset \mathbb{C} -P_c.$ Any pre-image of $x$ is determined by a univalent branch $\varphi$ of $\mathbf{f}_c^{-1}$ at $x.$ We are going to show that $\varphi(x)$ is in $ \alpha(z, \mathbf{f}_c).$ There is a backward orbit $z_{n_1} \to \cdots \to z_1 \to z$ such that $|x-z_{n_1}| < \epsilon_1.$ Define $z_{n_1 +1}$ as $\varphi(z_{n_1}).$ Since $$\alpha(\varphi(z_{n_1}), \mathbf{f}_c) = \alpha(z, \mathbf{f}_c),$$ there is another backward orbit $z_{n_2} \to \cdots \to \varphi(z_{n_1}) \to z_{n_1}$ with $|x-z_{n_2}| < \epsilon_2.$ In this way, we construct a backward orbit $(z_k)$  of $z$ such that $|z_{n_j} -x| < \epsilon_j$ and $z_{n_j +1}= \varphi(z_{n_j}).$ By taking limits, we conclude that $\varphi(x)$ belongs to $\alpha(z, \mathbf{f}_c).$  This proves that $\mathbf{f}_c^{-1}(\alpha(z, \mathbf{f}_c)) \subset \alpha(z, \mathbf{f}_c).$ For the other inclusion it suffices to show that any $\zeta$ in $\alpha(z, \mathbf{f}_c)$ has at least one image in $\alpha(z, \mathbf{f}_c).$ 
 Indeed, if $\zeta$ is in $\alpha(z, \mathbf{f}_c),$ then there is a  backward orbit $(z_i)_0^{\infty}$  of $z$  such that $\zeta = \lim_{n \in N_1} z_n,$ where $N_1$ is an infinite subset of $\mathbb{N}$ ($\lim_{n\in N_1}$ means the limit of the subsequence $(z_n)_{n\in N_1}$). Since $\zeta \in \mathcal{D}_c,$ we have $\zeta\neq 0$ and there are $q$ univalent branches $\phi_1, \cdots, \phi_q$ of $\mathbf{f}_c$ at $\zeta.$ By deleting finitely many $n\in N_1,$ if necessary, we may suppose that  $z_n$ is in the domain of $\phi_i,$ for every $i$ and $n \in N_1.$  If $n\in N_1$, then there is $k_n \in \{1, \cdots, q\}$  such that $\phi_{k_n}(z_{n}) = z_{n -1}.$ The sequence $(k_n)_{ n\in N_1}$ has an  element that repeats infinitely many times,  say $k_1=k_n,$ for every $n$ in an infinite set $N_2 \subset N_1.$ Hence $\phi_{k_1}(z_{n}) =z_{n-1}$ for every $n\in N_2.$ We conclude that $$\phi_{k_1}(\zeta)=\lim_{n\in N_2} \phi_{k_1}(z_n) =\lim_{n\in N_2} z_{n-1} $$ belongs to $\alpha(z, \mathbf{f}_c).$ Hence $\zeta$ has at least one image in $\alpha(z, \mathbf{f}_c),$ as desired.    

{Fourth claim: $J_c \subset \alpha(z,\mathbf{f}_c).$ } Indeed, let $(z_i)_0^n$  be a repelling cycle. Recall that $J_c$ is contained in $\mathbb{C}  - P_c$ when $c$ is in $H_{\beta}.$ By the second claim, every $z_i$ is in $ \alpha(z_0, \mathbf{f}_c) = \alpha(z,\mathbf{f}_c).$  Hence $\alpha(z,\mathbf{f}_c)$ contains all repelling cycles. Since $\alpha(z, \mathbf{f}_c)$ is closed,  $J_c$ is a subset of $\alpha(z, \mathbf{f}_c).$  

Using the first, third and fourth claims, we conclude that  $J_c \subset \mathbb{C} - P_c$ is a hyperbolic repeller. By Theorem \ref{lqe}, $J_c^*$ is a hyperbolic attractor for $\mathbf{g}_c.$  

 Let $U$ be an open set containing $J_c$ and $z\in \mathbb{C} -P_c.$ We want to show that there is $n$ such that $\mathbf{f}_c^{-n}(z)$ is contained in $U.$ Since $J_c$ is a hyperbolic repeller, $U$ contains a smaller  backward invariant neighbourhood of $J_c.$ It suffices to prove that $\mathbf{f}_c^{-n}(z)$ is contained in the smaller neighbourhood, which we denote by $U$ as well. Suppose the assertion $\mathbf{f}_c^{-n}(z) \subset U$ is not true for every $n.$  Since $U$ is backward invariant, for every $n$ there is backward orbit of $z$ with length $n$ which is  contained in the complement of  $U.$ Using Lemma \ref{pfv}
we extract an infinite backward orbit of $z$ in the complement of $U,$ with a limit point in $\mathbb{C} - J_c,$ contradicting $J_c =\alpha(z, \mathbf{f}_c).$
\end{Proof}

\begin{cor} \label{vmm} If $c$ is in $H_{\beta},$ then any $\mathcal{D}_c$  in the conditions of   Lemma \ref{vnm} satisfies \begin{equation}\label{fkk} J_c = \bigcap_{k>0} \mathbf{f}_c^{-k}(\mathcal{D}_c).\end{equation}
In particular, if $w$ has an infinite forward orbit contained in a compact set $E\subset \mathbb{C} - J_c^*,$ then $w\in J_c.$ \end{cor}

\begin{proof} Since $J_c$ is backward invariant and contained in $\mathcal{D}_c,$  we only need to prove that the nested intersection \eqref{fkk} -- temporarily denoted by $L$ --  is contained in $J_c.$  By Lemma \ref{pfv}, a point $z$ belongs to  $L$  iff   $z $ has an infinite forward orbit  contained in $\mathcal{D}_c.$ If $z$ has an infinite forward orbit in $\mathcal{D}_c$, we are going to show that $d(z, J_c) < \epsilon, $ for every $\epsilon >0,$ where $d$ denotes the hyperbolic distance function of $\mathbb{C} -P_c.$  Indeed, the orbit $(z_n)$ is bounded and has a limit point $z_*$ in $ \mathbb{C} -P_c.$  
By Theorem \ref{xtl}, every univalent branch of $\mathbf{f}_c^{-1}$ contract distances by a factor $\mu<1$ on $\mathcal{D}_c\Subset \mathbb{C}  - P_c.$  Let $w$ be a point of $J_c.$  Since $P_c'$ is either a single point set or a Cantor set, there is a bounded simply connected $U\subset \mathbb{C} -P_c$ containing $\{z_*, w\}.$
  By Lemma \ref{vnm}, there is $n_0$ such that $\mathbf{f}_c^{-n_0}(U) \subset \mathcal{D}_c.$  Since $z_*$ is a limit point, there is $n>2n_0$ such that $z_n \in U$ and $\mu^{n-n_0}\operatorname{diam}(\mathcal{D}_c) < \epsilon,$ where $\operatorname{diam}$ denotes diameter with respect to the hyperbolic metric of $\mathbb{C} -P_c.$

   By Lemma \ref{solf}, there is a backward orbit of $U,$
\begin{equation}\label{plmm}U_0 \xrightarrow{\varphi_0} U_1 \xrightarrow{\varphi_1} \cdots \xrightarrow{\varphi_{n-1}} U_n=U,\end{equation}  given by conformal isomorphisms $\varphi_i: U_i \to U_{i+1},$ with $z_i \in U_i \subset \mathbb{C} - P_c.$ Every $\varphi_i$ is a univalent branch of $\mathbf{f}_c.$ Let $(w_i)_0^n$ be the orbit determined by $w_i \in U_i,$ $\varphi_{i}(w_i) = w_{i+1}$ and $w_n=w.$      Since $n >2n_0,$ we have $$U_{n-n_0}\subset \mathbf{f}_c^{-n_0}(U) \subset \mathcal{D}_c.$$ Therefore, $$ \operatorname{diam}(U_0) \leq \mu^{n-n_0} 
\operatorname{diam} (U_{n-n_0}) \leq \mu^{n-n_0} \operatorname{diam}(\mathcal{D}_c) < \epsilon.$$

\noindent Since $J_c$ is backward invariant (Theorem \ref{lqp}), every point of $(w_i)$ is in $J_c.$ Both $z=z_0$ and $w_0$ belong to $U_0.$ Hence  $d(z_0, w_0) < \epsilon$  and $d(z, J_c) < \epsilon,$ as desired.

Now suppose $(z_i)$ is an infinite forward orbit contained in a compact set $E\subset \mathbb{C} - J_c^*.$  We want to show that every $z_i$ belongs to $J_c.$ Since every $z_i$ has a bounded orbit,  it follows that $z_i\in K_c.$  In particular, the whole sequence $(z_i)$ is contained in  $\mathbf{f}_c^{-n}(B_R),$ for every $n>0.$ For the case where $c$ is sufficiently close to a simple centre,  there exists a critical system $\mathcal{A}_c$ and no point of $(z_i)$ can be in $\operatorname{dom}(\mathcal{A}_c);$ indeed,  if some $z_k$ is in $\operatorname{dom}(\mathcal{A}_c),$ then  there is  a subsequence converging to a point of $J_c^*$ (see Theorem  \ref{ste}), contradicting the fact that $z_i$ never escapes $E.$  

Recall from the proof of Theorem \ref{vnm}   that,  if $c$ is sufficiently close to  a simple centre, then  $\mathcal{D}_c = \mathbf{f}_c^{-n}(B_R) - \operatorname{dom}(\mathcal{A}_c)$. If $c$ is not in $M_{\beta, 0}$ then $\mathcal{D}_c = \mathbf{f}_c^{-n}(B_R),$ for some $n>0.$  In any case,  the sequence $(z_i)$ is contained in $ \mathcal{D}_c.$  Since we have already proved that \eqref{fkk}  holds, it follows that every point of the sequence $(z_i)$ is in $ J_c.$
\end{proof}

\begin{thm}\label{ght} If $c$ is not in  $M_{\beta,0},$ then $J_c=K_c$ and   the dual Julia set is empty. In this case, for any point  $z$ in the complement of $J_c$ and any $\epsilon$-neighbourhood $U$ of $\infty,$ {there exists} $n_0$ such that    $\mathbf{f}^{n}_c(z) \subset U,$  for every $n>n_0.$ 

\end{thm}

\begin{proof} By Lemma \ref{lkc}, if $c$ is not in $M_{\beta,0},$ then some backward iterate  $A=\mathbf{f}_c^{-n_0} (B_R) $ is contained in $ \mathbb{C} - P_c.$  Since $A$ is bounded, by Lemma \ref{vnm} there is $n_1$ such that $\mathbf{f}_c^{-n_1}(A)$ is contained in  $\mathcal{D}_c.$   Let $k_1 = n_0 + n_1.$ Since $\mathbf{f}_c^{-k_1}(B_R)$ is contained in  $\mathcal{D}_c,$ from Corollary \ref{vmm} we have
$$K_c = \bigcap_{n> k_1} \mathbf{f}_c^{-n}(B_R) = \bigcap_{n>0} \mathbf{f}_c^{-n}(\mathbf{f}_c^{-k_1}(B_R)) \subset J_c. $$

\noindent By Theorem \ref{lqp} the Julia set is backward invariant and, therefore, $J_c \subset K_c.$ Hence $J_c= K_c.$

Let $U$ be a neighbourhood of $\infty$ and $z\in \hat{\mathbb{C}}  - J_c.$  We may suppose, without loss of generality, that $U$ is forward invariant.  If  no iterate $\mathbf{f}_c^{n}(z)$ is contained in $U,$ then for every $n$ there corresponds a forward orbit of $z$ with length $n$ which is in the complement of $U.$ Using Lemma \ref{pfv} we extract an infinite forward orbit of $z$ which is in the complement of $U.$ In particular, $z$ has a bounded forward orbit and we conclude that $z\in K_c=J_c,$ which contradicts the assumption.  \end{proof}

\subsection{Hyperbolicity.} Hyperbolic rational maps play a central role in holomorphic dynamics and can be defined in several equivalent ways. We shall present a dynamical definition first and then prove that any parameter in $H_{\beta}$ is hyperbolic.

 The basin of attraction of $J_c^{*}\cup \{\infty\}$  consists of every point $z$ of the Riemann sphere such that $\omega(z, \mathbf{f}_c) \subset J_c^{*} \cup \{\infty\}.$

From the dynamical point of view, hyperbolic correspondences are those which display simplest possible behaviour, in which $J_c$ is \emph{source} and $J_c^* $ is a \emph{sink}. More precisely: 

\begin{defi}[Hyperbolicity] We say that $\mathbf{f}_c$ is hyperbolic, or that $c$ is a hyperbolic parameter, if the basin of attraction of $J_c^*\cup \{\infty\} $  is $\hat{\mathbb{C}}- J_c.$ 

\end{defi}

Both $J_c$ and $J_c^*$ are contained in the filled Julia set $K_c$. For the quadratic family, for example, $J_c$ is the boundary of $K_c.$  Therefore, understanding the dynamics of $\mathbf{g}_c: K_c \to K_c$ is the first step to show that $\mathbf{f}_c$ is hyperbolic when $c\in H_{\beta}.$ Indeed, under the action of $\mathbf{g}_c,$ any orbit of a point in $K_c  - J_c$ is attracted to $J_c^*$:

\begin{thm} \label{mcv}Suppose $c$ is sufficiently close to a simple centre. Then $J_c^*$ is nonempty and is contained in the interior of $K_c.$ The Julia set  $J_c$ is properly contained in  $ K_c.$ For every $z\in K_c - J_c,$  \begin{equation}\label{gegsdw} \omega(z, \mathbf{g}_c) = J_c^*,\end{equation} and for any open set $U\supset J_c^*,$ there is $n_0>0$ such that $\mathbf{g}^n_c(z)\subset U,$  \ for every $n>n_0.$ In particular, forward orbits are dense in $J_c^*.$
\end{thm}

\begin{proof} Since $J_c^*$ is contained in $\operatorname{dom}(\mathcal{A}_c) \subset K_c$  (see Theorem \ref{wqt}),  it follows that $J_c^*$ is contained in the interior of $K_c.$ Since $J_c$ is contained in the complement of $\operatorname{dom}(\mathcal{A}_c),$ we have $J_c\subsetneq K_c.$

Let $U\supset J_c^*$ be an open set and $z\in K_c - J_c.$ We are going to show that there is $n>0$ such that $\mathbf{g}_c^n(z) \subset U.$ Since $\operatorname{dom}(\mathcal{A}_c)$ is forward invariant under $\mathbf{g}_c,$ there is another open set  $U_1 \subset U$ containing $J_c^*$ which is forward invariant under $\mathbf{g}_c.$ It suffices to show that $\mathbf{g}_c^{n}(z) \subset U_1,$ for some $n>0.$ Indeed, if no $\mathbf{g}_c^{n}(z)$ is contained in $U_1,$ then 
for every $n$ there is a finite forward orbit $z_0 \to \cdots \to z_n$  of $\mathbf{g}_c$ with $z_0=z$ and $z_n \not \in U_1.$ Since $U_1$ is forward invariant under $\mathbf{g}_c,$ the whole sequence is in the complement of  $U_1.$  Using Lemma \ref{pfv},  we extract an infinite forward orbit $(z_i)$ under $\mathbf{g}_c$ which is in $\mathbb{C} - U_1.$ Since $(z_i)$ is an orbit of $\mathbf{g}_c,$ we have $z_i\in K_c,$ for every $i.$ Thus $(z_i)$ is an infinite forward orbit contained in the compact set $E=K_c - U_1.$ By Lemma \ref{vmm}, the sequence $(z_i)$ is contained in   $J_c,$ which is a contradiction, since $z_0$ is in $K_c - J_c.$ 

Equation \eqref{gegsdw} follows from Theorem \ref{rrr}.   \end{proof}

{In the following corollary, recall from Definition \ref{sdboidg} that $H_{\beta}$ is the union of the complement of $M_{\beta,0}$ with a sufficiently small neighbourhood of the set of all simple centres. }

\begin{cor}\label{dgeqd}   If $c$ is in $H_{\beta},$ then $\mathbf{f}_c$ is hyperbolic.  \end{cor}

\begin{proof} Since any forward orbit converges exponentially fast to $\infty$ on $\hat{\mathbb{C}} - K_c,$  the corollary is a consequence of Theorems \ref{mcv} and \ref{ght}. 
\end{proof}
\subsection{Sensitive dependence on initial conditions.} The Julia set $J(f)$ of a rational function $f$ satisfies the following property: for any open set $U$ intersecting $J,$ there is $n>0$ such that $$f^n(U\cap J) =J.$$  In general, any mapping  $g: X \to X$ from a topological space that eventually maps every open set $U$ onto the whole space, $g^n(U) = X,$ is called \emph{locally eventually onto}, or LEO. 
This property implies sensitive dependence on initial conditions, leading to a very precise characterisation of \emph{chaos}. 

It is the purpose of this section to prove a similar result for the Julia set of $\mathbf{f}_c$ when $c \in H_{\beta}.$ Roughly speaking, \ this means that the interesting dynamics of $\mathbf{f}_c: \hat{\mathbb{C}} \to \hat{\mathbb{C}}$  is concentrated in $J_c$ when $c\in H_{\beta}.$ We shall see that  $J_c$ is stable by means  of holomorphic motions.

If $\Lambda$ is a hyperbolic repeller of $\mathbf{f}_c,$ then the restriction $\mathbf{f}_c|_{\Lambda}(z)= \mathbf{f}_c(z) \cap \Lambda$ is a well defined multifunction from $ \Lambda$ to $ \Lambda$  (since every point of  $\Lambda$ has at least one image in $\Lambda).$  We say that $\Lambda$ is \label{sdfqeg}\emph{LEO} (with respect to $\mathbf{f}_c$), or locally eventually onto, if for every open set $U$ intersecting $\Lambda,$ there is $n>0$ such that \begin{equation} \label{wtc}\mathbf{f}^n_c|_{\Lambda}(U\cap \Lambda) = \Lambda.\end{equation}
Since the complement of a hyperbolic repeller is forward invariant under $\mathbf{f}_c,$ it follows that \eqref{wtc}  is equivalent to $\mathbf{f}_c^n(U) \supset \Lambda.$

\begin{thm} \label{luf} If $c\in H_{\beta},$ then $J_c$ is a LEO hyperbolic repeller of $\mathbf{f}_c.$  

\end{thm}

\begin{proof}  Due to Theorem \ref{lqp}, we only need to show that $J_c$ is LEO. Let  $U$ be an open set intersecting $J_c.$ Let $d_c$ denote the  hyperbolic distance function of $\mathbb{C} -P_c$ (Theorem \ref{xtl}).  
If $z\in J_c \cap U,$ then some ball $B(z,\delta)$ of radius $\delta>0$ and centre $z$ with respect to $d_c$ is contained in $U.$  Consider a forward orbit $(z_i)$ of $z$  contained in $J_c.$  (Recall that every point of $J_c$ has at least one image in $J_c$). 
Since the sequence $(z_n)$ is bounded, it has a limit point $z_* \in J_c.$  For every $x\in J_c,$ the set $\{x, z_*\}$ is contained in  $\mathbb{C} -P_c \supset J_c.$ Recall from  Theorem \ref{dstg} that 
$P_c'$ is either a single point set or a Cantor set. Hence,  there is a bounded simply connected set $V_{x} \supset \{x,z_*\}$ such that $V_{x} \subset \mathbb{C} - P_c.$  The open cover $\{V_x\}_{x\in J_c}$ of the compact set $J_c$ has a finite sub-cover $\{V_{x_i}\}_{i=1}^{k}$ with $k$ open sets.  
According to Lemma \ref{vnm}, there is a backward invariant  open set $\mathcal{D}_c \Subset \mathbb{C}  - P_c$ such that $J_c \subset \mathcal{D}_c.$

By Theorem \ref{xtl} and Lemma \ref{vnm}, there are $\mu <1$  and $n_0>1$ such that every univalent branch of $\mathbf{f}_c^{-1}$ defined on a subset of $\mathcal{D}_c$ contracts the hyperbolic metric by the factor $\mu,$ and $$k \mu^{n_0} \operatorname{diam}_c(\mathcal{D}_c) < \delta/2   \ \ \textrm{and}  \ \   \bigcup_{i=1}^k \mathbf{f}_c^{-n_0}(V_{x_i}) \subset \mathcal{D}_c,$$  where $\operatorname{diam}_c$ denotes the diameter with respect to $d_c.$  
Let $X= \bigcap_{i=1}^{k} V_{x_i}.$

 Since $z_*$ belongs to $X,$  there is  $n_1 > 2n_0$ such that $z_{n_1} \in X.$ By Lemma \ref{solf}, every $V_{x_i} \subset \mathbb{C} - P_c$ determines a sequence of maps
 $$U_{0,i} \xrightarrow{\varphi_{1,i}} U_{1,i} \xrightarrow{\varphi_{2,i}} \cdots \xrightarrow{\varphi_{n_1,i}} U_{n_1,i}:= V_{x_i}, $$
where each $\varphi_{j,i}: U_{j-1,i} \to U_{j,i}$ is a conformal isomorphism (and also a branch of $\mathbf{f}_c$), $z_j \in U_{j,i},$ and $\varphi_{j,i}$ sends $z_{j-1}$ to  $z_j.$ (This is possible because $z_{n_1} \in V_{x_i}$).  In particular, $z=z_0$ belongs to $U_{0,i},$ for every $i.$

\noindent Let $W= \bigcup_{i} U_{0,i}.$  Since $n_1 > 2n_0$ and $\mathbf{f}_c^{-n_0}(V_{x_i})$ is contained in $\mathcal{D}_c,$ we have $U_{n_0,i} \subset  \mathcal{D}_c.$  
Thus, 
$$\operatorname{diam}_c(U_{0,i}) \leq \mu^{n_0} \operatorname{diam}_c(U_{n_0,i}) \leq \mu^{n_0} \operatorname{diam}_c(\mathcal{D}_c);$$ and since $\bigcap_i U_{0,i}$  is nonempty, 
$$\operatorname{diam}_c(W) \leq \sum_{i=1}^{k} \operatorname{diam}_c(U_{0,i}) \leq k \mu^{n_0} \operatorname{diam}_c(\mathcal{D}_c) < \delta/2. $$ 
It turns out that $W\subset B(z, \delta) \subset U, $ and $$\mathbf{f}_c^{n_1}(U)  \supset \mathbf{f}_c^{n_1}(W) \supset \bigcup_{i} V_{x_i} \supset J_c.$$
Therefore, $J_c$ is a LEO hyperbolic repeller. 
\end{proof}

\subsection{Holomorphic motions.}\label{gvn}

Holomorphic motions arise naturally in the study of structurally stable rational maps.  For holomorphic correspondences on the plane and holomorphic maps in higher dimensions  we need a new definition which allows \emph{branches.} This leads to the definition of branched holomorphic motion originally introduced by Dujardin and Lyubich  \cite{Lyubich2015} for dissipative polynomial automorphisms of $\mathbb{C}^2.$ The following definition is an adaptation for holomorphic correspondences given in \cite{SS17}.

\begin{defi}[Branched holomorphic motion] \label{gehgds}\normalfont 
 Let $\Lambda$ and $U$ be nonempty subsets of $\mathbb{C}$ and suppose $U$ is open.   A multifunction $\mathbf{h}: U\times \Lambda \to \mathbb{C}$ is a branched holomorphic motion with base point $a$ in $U$ if
 (i) $\mathbf{h}_a(z) =\{z\},$ for every $z\in \Lambda;$ and (ii) $$ \bigcup_{z \in \Lambda} \left \{(c,w): c\in U, w\in \mathbf{h}_c(z)   \right \}  = \bigcup_{f\in \mathcal{F}} G_{f}$$ where $\mathcal{F}$ is a family of holomorphic functions $f:U \to \mathbb{C}$ and $G_{f}$ is the graph of $f.$ (Notice that the left-hand side of the above equation is a union of graphs of  multifunctions given by $c\mapsto \mathbf{h}_c(z).$) The holomorphic motion is said to be normal if $\mathcal{F}$ is normal.
 
   \end{defi}

\begin{thm}[\sc Structural stability]  \label{plmd}For every parameter $c_0$ in $H_{\beta}$  there is a normal branched holomorphic motion $$\mathbf{h}: U \times J_{c_0} \to \mathbb{C}$$ with base point $c_0$ such that $\mathbf{h}_c(J_{c_0})=J_c, $
for every $c\in U.$
\end{thm}

{The dual Julia set and the Julia set move continuously for parameters in $H_{\beta}$ (see Theorem \ref{jst}). Since $\mathbf{h}_c$ is multivalued and $J_{c}$ is the image of $J_{c_0}$ under $\mathbf{h}_c,$  the geometry of $J_c$ is likely to change drastically as we move $c $ close to $c_0,$  and therefore we cannot expect $J_c$ to be homeomorphic to $J_{c_0}.$ In particular, the dynamics restricted to the Julia sets cannot be topologically conjugate.  However, $\mathbf{h}_c$ is the projection of a holomorphic motion in $\mathbb{C}^2$ given by a family of conjugacies in the usual sense, see \cite{SS17}.  }

  Every hyperbolic repeller (such as $J_c$ when $c\in H_{\beta}$) is a particular example of \emph{hyperbolic set}, as defined in  \cite{SS17}.  According to Theorem 3.4 of \cite{SS17}, hyperbolic Julia sets of $\mathbf{f}_c$ move holomorphically, in the sense that if $J_{c_0}$ is a hyperbolic set, then there exist  a normal branched holomorphic motion $\mathbf{h}: U \times J_{c_0} \to \mathbb{C}$ with base point $c_0,$ and  an open set $\Omega$ containing $J_{c_0}$  such that $$\mathbf{h}_c(J_{c_0}) = \Omega \cap J_c, $$ for every $c\in U.$ In addition, if $c\mapsto J_c$ is continuous at $c_0,$ then $J_c$ remains within $\Omega$ for every $c$ sufficiently close to $c_0.$ In this way, Theorem \ref{plmd} is a corollary  of the following result. 
 
\begin{thm}[Stability] \label{jst} The set functions $c\mapsto J_c^*$ and $c\mapsto J_c$ are continuous at every $c \in H_{\beta}.$ 
\end{thm}
The proof of Theorem \ref{jst} will be given in the remaining pages of this paper, after the following lemma (its proof can be skipped at a first reading). 
   
   \begin{lem}\label{ncv} If $c_0$ is in $H_{\beta}$, then there exist   a neighbourhood $V$ of $c_0,$ a compact set $E,$  and $\epsilon_0>0$ such that \begin{equation} \label{tes}J_c \subset \operatorname{int}(E) \subset (E)_{3\epsilon_0} \subset \mathbb{C} - P_c, \end{equation}
for every $c\in V,$ where $(E)_{3\epsilon_0}$ is the set of all $z$ such that $d(z, E) < 3\epsilon_{0}$ and $d$ is the Euclidean distance.

\end{lem}

   \begin{proof}\label{sdt}
By Lemma \ref{rvb}, we can fix an escaping radius $R$ such that $R -1$ is also an escaping radius of $\mathbf{f}_c,$ for every $c $ in a neighbourhood $V_0$ of $c_0.$  By taking $R$ sufficiently large, we may  suppose that  \begin{equation} \label{ifi} |z| \geq R -1 \implies |\mathbf{f}_c(z)| > R, \end{equation} for every $c\in V_0.$ As usual, by \label{huiji}$|\mathbf{f}_c(z)| > R $  we mean that $|w| >R,$ whenever $w\in \mathbf{f}_c(z).$ 

\parr{Step 1 (sandwich).}  We are going to show that: {\it for every $m>1,$ there is a neighbourhood $V_m \subset V_0$ of $c_0$  such that \begin{equation}\label{pwq} \mathbf{f}_c^{-m-2}(B_{R}) \subset \mathbf{f}_{c_0}^{-m-1}(B_{R}) \subset \mathbf{f}_c^{-m}(B_{R}), \end{equation}
for every $c\in V_m.$ }
\vspace{0.2cm}

\noindent Since $\mathbf{f}_{c_0}^{-m-2}(B_{R}) \Subset \mathbf{f}_{c_0}^{-m-1}(B_{R})$ there exists $\epsilon >0$ such that  $$(\mathbf{f}_{c_0}^{-m-2}(B_{R}))_{\epsilon} \subset \mathbf{f}_{c_0}^{-m-1}(B_{R}).$$ The set $E= \hat{\mathbb{C}} - (\mathbf{f}_{c_0}^{-m-2}(B_{R}))_{\epsilon}$ is compact.  Since $E$ is compact and contained in the complement of $(\mathbf{f}_{c_0}^{-m-2}(B_{R}))_{\epsilon},$  we have $|\mathbf{f}_{c_0}^{m+2}(z)| >R,$ for every $z\in E.$  The following is a standard compactness argument: by continuity, for every point $z$ of $E$ there are a neighbourhood $U_z$  of $z$ and a neighbourhood $W_z$ of $c_{0}$ such that $|\mathbf{f}_c^{m+2}(w)| > R,$ for every $(c,w)$ in $W_z \times U_z.$ Since $E$ is compact, there is a finite cover $\{U_{z_i}\};$ and if $V_m$ is the intersection of $V_0$ with $\cap_{i} W_{z_i},$ then $|\mathbf{f}_c^{m+2}(z)| >R,$ for every $(c,z)$ in $V_m\times E.$ In other words, $E \subset \hat{\mathbb{C}} - \mathbf{f}_c^{-m-2}(B_{R})$,
 for every $c\in V_m.$ By taking complements,
 $$ \mathbf{f}_c^{-m-2}(B_{R}) \subset \hat{\mathbb{C}} - E= \mathbf{f}_{c_0}^{-m-2}(B_{R}) \subset \mathbf{f}_{c_0}^{-m-1}(B_{R}), \ \ c\in V_m.$$
 
 This proves the first inclusion of \eqref{pwq}.
 
 We shall need the following property: since $R$ is an escaping radius of every $\mathbf{f}_c$ with $c\in V_0,$ 
 \begin{quote} \hypertarget{P}{$(P)$} the set  $A=\{|z| > R\}$ is contained in $$\{z: |\mathbf{f}_c^{m+1}(z)| \geq R\}=\{ z: |w| \geq R   \ \textrm{whenever $w\in \mathbf{f}_c^{m+1}(z)$} \},  $$ for every $c\in V_m.$ 
 \end{quote}
It follows that $\{z: |\mathbf{f}_c^{m}(z)| \geq R\} - A$ is contained in  the compact set $E_1=\{|z| \leq R\},$ for every $c\in V_m.$ By continuity, given $\epsilon>0$ we may reduce $V_m$ so that \begin{equation}\label{idi} d_H(\mathbf{f}_c^{m}(z), \mathbf{f}_{c_0}^m(z)) < \epsilon,\end{equation} whenever $c\in V_m$ and $z\in E_1,$ where $d_H$ denotes Hausdorff distance. (Here we have used another compactness argument to prove uniform continuity).  Take $\epsilon=1$ and let $c\in V_m.$  We want to show that $\{|\mathbf{f}_c^{m}(z)| \geq R\}$ is contained in $\{|\mathbf{f}_{c_0}^{m+1}(z)| \geq R\};$ equivalently, $\mathbf{f}_{c_0}^{-m-1}(B_{R})  \subset \mathbf{f}_c^{-m}(B_{R}),$ which completes Step 1. 
   Indeed, if  $z$ is in $\{|\mathbf{f}_c^{m}(z)| \geq R\} - A,$ then  $z \in E_1, $ and by  \eqref{idi} we have $|\mathbf{f}_{c_0}^{m}(z)| \geq R -1.$ From \eqref{ifi} we conclude that $|\mathbf{f}_{c_0}^{m+1}(z)| >R.$ On the other hand, if  $z$ is in $\{|\mathbf{f}_c^{m}(z)| \geq R\} \cap A,$  then by \hyperlink{P}{$(P)$} the point $z$ is also in $\{|\mathbf{f}_{c_0}^{m+1}(z)| \geq R\}.$ 
  
  Step 1 is completed. 
  
  \parr{Step 2: prove \eqref{tes} for $c_0$ close to a simple centre. } Let $a$ be a simple centre. Using Theorem \ref{fed} we construct four critical systems  $\mathcal{A}_c^i,$ with $1\leq i\leq 4,$ parameterised in a neighbourhood $U$ of $a,$ such that
  \begin{equation}\label{xtb}
  \operatorname{dom}(\mathcal{A}_{c_4}^4) \subset 
  \overline{\operatorname{dom}(\mathcal{A}_{c_3}^{3})} \subset \operatorname{dom}(\mathcal{A}_{c_2}^2) \subset \operatorname{dom}(\mathcal{A}_{c_1}^{1}),
  \end{equation}
  \noindent for every $c_1, c_2, c_3$ and $c_4$ in $U.$

According to Theorem \ref{fed}, there  is $m>1$ such that $\mathbf{f}_c^{-m}(B_{R})$ is a region satisfying the conditions $(c)$ and $(d)$ of Theorem \ref{ppp} with respect to all critical systems $\mathcal{A}_c^i$, for any $i$ and $c\in U.$  Hence, the iterates of the critical point either remain in $\operatorname{dom}(\mathcal{A}_c^i)$ or are sent to the complement of $\mathbf{f}_c^{-m}(B_R).$ In particular, $\mathbf{f}_c^{-m}(B_{R}) - \operatorname{dom}(\mathcal{A}_c^4)$ is contained in $\mathbb{C} - P_c.$

According to Step 1, there is a neighbourhood $V_m \subset U$ of $c_0$ such that $\mathbf{f}_{c_0}^{-m-1}(B_{R}) \subset \mathbf{f}_c^{-m}(B_{R}),$ for every $c\in V_m.$ If $n>m,$ then $\mathbf{f}_{c_0}^{-n-1}(B_{R}) \Subset \mathbf{f}_{c_0}^{-m-1}(B_{R}).$ In view of Step 1, we may reduce $V_m$ so that $\mathbf{f}_c^{-n-2}(B_{R})\subset \mathbf{f}_{c_0}^{-n-1}(B_{R}),$ for every $c$ in $V_m.$ By Lemma \ref{wqt},

\begin{equation}\label{xx}
\begin{split}
J_c \subset \mathbf{f}_c^{-n-2}(B_{R}) - \operatorname{dom}(\mathcal{A}_c^1) &\subset \mathbf{f}_{c_0}^{-n-1}(B_{R}) - \operatorname{dom}(\mathcal{A}_{c_0}^2)\\
 & \Subset \mathbf{f}_{c_0}^{-m-1}(B_{R})- \overline{\operatorname{dom}(\mathcal{A}_{c_0}^{3})} \\
 & \subset \mathbf{f}_c^{-m}(B_{R}) - \operatorname{dom}(\mathcal{A}_c^{4}) \subset \mathbb{C} -P_c,
\end{split}
\end{equation}
for every $c\in V_m.$  If $E$ is the closure of  $\mathbf{f}_{c_0}^{-n-1}(B_{R}) - \operatorname{dom}(\mathcal{A}_{c_0}^2),$ then $E$ satisfies \eqref{tes} for some $\epsilon_0.$ 

\parr{Step 3: prove \eqref{tes} when $c_0$ is not in $M_{\beta,0}.$ }   According to Lemma \ref{rvb}, there is an escaping radius $R$  that holds for  every $\mathbf{f}_c$ with $c$ sufficiently close to $c_0.$ Since $|\mathbf{f}_{c_0}^{m}(0)| > R,$ for some $m>1,$ by continuity we conclude that $|\mathbf{f}_c^m(0)| >  R,$ for every $c$ in a neighbourhood $V$ of $c_0.$  This implies $\mathbf{f}_c^{-m}(B_{R}) \subset \mathbb{C} -P_c,$ for if  $z$ belongs to $P_c \cap \mathbf{f}_c^{-m}(B_{R}),$ then in particular it belongs to $P_c \cap B_R;$
 and since $R$ is an escaping radius, $z$ must be in $\mathbf{f}_c^{i}(0),$ for some natural number $i<m.$ Hence $|\mathbf{f}_c^m(z)| >R,$ which contradicts $z\in \mathbf{f}_c^{-m}(B_{R}).$  By Step 1, we can find $n>m+1$ and replace $V$ by a smaller neighbourhood of $c_0$  such that, for every $c\in V,$
$$J_c \subset \mathbf{f}_c^{-n-1}(B_{R})\subset \mathbf{f}_{c_0}^{-n}(B_{R}) \Subset \mathbf{f}_{c_0}^{-m-1}(B_{R}) \subset \mathbf{f}_c^{-m}(B_{R}) \subset \mathbb{C} -P_c. $$

This  proves \eqref{tes} with $E= \overline{\mathbf{f}_{c_0}^{-n}(B_{R})}.$

The proof of Lemma  \ref{ncv} is complete.  \end{proof}

\begin{Proof}{of Theorem \ref{jst}}
Since $a$ is a simple centre, there is only one bounded orbit of zero under $\mathbf{f}_a,$ and this orbit is a cycle $(z_i)_0^n.$  For every $c\neq a$ in a neighbourhood $V$ of $a,$ $$J_c^* =\Lambda(\mathcal{F}_c) = \bigcup_{i=0}^{n-1} \Lambda_{i,c},$$   where each $\Lambda_{i,c}$ is a Cantor set with $\Lambda_{i,c} \to \{z_i\}$ as $c\to a$ (in the Hausdorff topology). This follows from Theorem \ref{lqe}, equation \eqref{rcv} and Theorem \ref{ppp}. Since $J_a$ is just the cycle $(z_i)_0^n$ containing zero, we conclude that $c\mapsto J_c^*$ is continuous at $c=a.$ The continuity at points $c_0 \in V-\{a\}$ follows from $J_c^*=\Lambda(\mathcal{F}_c)$ and  Theorem \ref{pcvm} $(c)$.

By Lemma \ref{lkc}, if $c_0$ is in $\mathbb{C} - M_{\beta,0}$ then $J_c^*$ is empty on a neighbourhood of $c_0$ and the map $c \mapsto J^*_c$ is clearly  continuous at $c_0.$ We conclude that $c\mapsto J_c^*$ is continuous at every point of $H_{\beta}.$  

\vspace{0.2cm}

 {\it On the continuity of  $c\mapsto J_c.$ } The continuity of $c\mapsto J_c$ at $c_0=0$ is a delicate subject, since $\mathbb{C} - P_0$ is not hyperbolic. It must be handled separately. Fortunately, this proof has already been presented in \cite[Theorem 4.1]{SS17}. 
From now on, assume $c_0\neq 0.$

\emph{Step 1: preliminary assumptions.} Suppose $c_0$ is in $H_{\beta}.$ Let $E$ be a compact set satisfying the properties stated in Lemma \ref{ncv}. For every $c$ in a neighbourhood of $c_0,$ let $d_c$ denote the hyperbolic distance function of $\mathbb{C} - P_c.$ Since any two conformal metrics are equivalent on compact sets, there is a constant $C>1$ and a neighbourhood $V$ of $c_0$ such that 
\begin{equation}\label{htc} d_c(z, z+ w) \leq C|w|, \end{equation} for every $z$ in $(E)_{\epsilon_0},\  c\in V,$ and $|w| \leq \epsilon_0.$ (Recall that $(A)_{\delta}$ is the set of all points which are within distance $\delta$ of $A,$ with respect to the Euclidean metric.)
 We may reduce $V$ if necessary and consider  a constant $C_1>1$ such that, for every $c\in V,$  
 
  \begin{equation} \label{stu}\{z: d_c(z, E) < 2\epsilon_0/C_1\} \subset (E)_{2\epsilon_0}\end{equation} and 
 \begin{equation} \label{stf}
 C_1^{-1} d_e(z,w) \leq d_c(z,w) \leq C_1 d_e(z,w),  \ \ z,w\in (E)_{2\epsilon_0},
 \end{equation}
where $d_e$ denotes the Euclidean distance. Using Theorem \ref{xtl} and a compactness argument, we find $\mu <1 $ such that, for any $c\in V$ and any holomorphic branch $\varphi: U \to \mathbb{C}$ of $\mathbf{f}_c^{-1}$ with $U\subset (E)_{2\epsilon_0},$ we have 
 \begin{equation} \label{xrs} d_c(\varphi(z), \varphi(w)) < \mu d_c(z,w),   \end{equation}
 for any $z$ and $w$ in $U.$ 

From Theorem \ref{lqp}, the Julia set $J_{c_0}$ is a hyperbolic repeller and  any sufficiently small neighbourhood $U_{c_0}$ of $J_{c_0}$ is backward invariant under $\mathbf{f}_{c_0}.$  Since $J_{c_0}$ is contained in the interior of $E,$ we may fix one such neighbourhood  $U_{c_0} \subset E.$

\noindent Fix $z\in U_{c_0}.$ By Theorem \ref{lqp} , $\alpha(z, \mathbf{f}_{c_0})= J_{c_0}$. Let $(z_i)_0^{\infty}$ be a backward orbit of $z$ under $\mathbf{f}_{c_0}$. Hence $z_0 =z$ and $\mathbf{f}_{c_0}$ sends every $z_{i+1}$ to $z_i$. Since $U_{c_0}$ is backward invariant, every $z_i$ belongs to $E.$

  Let $B_i$ denote the ball centred at $z_i$ of radius $2\epsilon_0/C_1$ with respect to the metric $d_{c_0}.$ Since  $d_{e} \leq C_1 d_{c_0},$  the ball $B_i$  is contained in $(E)_{2\epsilon_0}. $   Hence $B_i$ is contained in $ \mathbb{C} - P_{c_0},$ and for every $i$ there is a univalent branch $\varphi_i$ of $\mathbf{f}_{c_0}^{-1}$ defined on $B_i$ sending $z_i$ to $z_{i+1}.$  Since $\varphi_{i}$ is a contraction, $\varphi_i(B_i)$ is contained in $ B_{i+1}.$

The main idea of the proof is given in the following two steps. Given $0< \epsilon < \epsilon_0$, we shall prove that if $c,c_0\in V$ satisfy
\begin{equation}\label{llt} |c-c_0| \sum_{i=0}^{\infty} \mu^{i} C < \epsilon/C_1, \end{equation} then $d_H(J_c, J_{c_0}) < \epsilon,$ where $d_H$ is the Hausdorff distance obtained from the Euclidean metric.  Equivalently, we shall prove that $J_{c_0} \subset (J_c)_{\epsilon}$ and $J_c \subset (J_{c_0})_{\epsilon},$ establishing the continuity of the Julia set at every point of $H_{\beta}.$

 \emph{Step 2: shadowing.} We are going to define a backward orbit $(w_i)_0^{\infty}$ of $z$ under $\mathbf{f}_c$ satisfying the following conditions: the sequence $(w_i)$ must be contained in  $(E)_{\epsilon_0}$ and  \begin{equation} \label{lgbm} d_{c_0}(w_{i}, z_{i}) < \mu C|c-c_0| + \mu^2C|c-c_0| + \cdots + \mu^iC|c-c_0|, \end{equation}
  for every $i>0.$   The proof will be given by induction. Let $w_0=z.$ In order to define $w_1,$ we shall use the  following fact: for any branch $\varphi$  of $\mathbf{f}^{-1}_{c_0},$ the map   $$z \mapsto \varphi(z-c+c_0)$$ is a branch of $\mathbf{f}^{-1}_c$.   By  \eqref{htc} and \eqref{llt}, the distance $d_{c_0}(w_0 -c +c_0,  w_0)$ is not greater than $C|c-c_0| < \epsilon/C_1.$  Therefore, $w_0-c+c_0$ is in the domain of $\varphi_0$ and we may define $w_{1}$ as $\varphi_0(w_0 -c+c_0).$ It follows that  $\mathbf{f}_c$ sends $w_1$ to $w_0.$    
Since $\varphi_0$ contract distances  on $(E)_{2\epsilon_0},$ from \eqref{xrs}  and \eqref{llt} we have 
$$d_{c_0}(w_{1}, z_{1}) < \mu d_{c_0}(w_0 -c+c_0, w_0) \leq \mu C |c-c_0| < \epsilon/C_1. $$

\noindent Since $d_e\leq C_1 d_{c_0}$ and $z_1 \in E,$ the point $w_1$ belongs to $(E)_{\epsilon}.$
 Hence \eqref{lgbm} is true for  $i=1.$ If \eqref{lgbm} is true  for $i=n,$ then  \begin{equation}\label{orbn} d_{c_0}(w_n -c+c_0, z_n) \leq C|c-c_0| + d_{c_0}(w_n, z_n) \leq \sum_{j=0}^{n} \mu^j C|c-c_0| < \frac{\epsilon}{C_1},  \end{equation}

\noindent and so $w_n -c+c_0$ is in $B_n.$ Since the domain of $\varphi_n$ contains $B_n$, we may define $w_{n+1}$ by $\varphi_n(w_n -c+c_0).$  It follows from \eqref{orbn} that  $$d_{c_0}(w_{n+1}, z_{n+1}) < \mu d_{c_0}(w_n -c+c_0, z_n) \leq  \sum_{j=1}^{n+1}\mu^j C|c-c_0| < \epsilon/C_1.$$ Since $d_e \leq C_1 d_{c_0},$ the point $w_{n+1}$ is in $(E)_{\epsilon_0}.$ 
This completes the induction process and Step 2.

\emph{Step 3: continuity.} It follows from Step 2 that, for  every limit point $z_*$ of $(z_j)$  there is a limit $w_*$ of  a subsequence of $(w_j)$ such that $d_{c_0}(w_*, z_*) < \epsilon/C_1.$ Since $(z_j)$ is arbitrary, and taking into account that $\alpha(z,\mathbf{f}_{c_0}) = J_{c_0}$ and $\alpha(w_0, \mathbf{f}_c) = J_c,$ we conclude that $$J_{c_0}  \subset \{z\in \mathbb{C}: d_{c_0}(z, J_c) < \epsilon/C_1 \} \subset (J_c)_{\epsilon}.$$
Starting with $c$ instead of $c_0,$ the  same argument of Step 2 can be applied to construct a shadowing  of  any backward orbit under $\mathbf{f}_{c}.$   In the same way, we obtain $J_c \subset (J_{c_0})_{\epsilon}.$  

The proof of Theorem \ref{jst} is complete.   
\end{Proof}

\appendix

\section{List of symbols}

\begin{enumerate}[noitemsep, label=$\alph*.$]
       \item  The exponent   $\beta$  of the family $\mathbf{f}_c$  is a rational number $\frac{p}{q}>1$ (see  page   \pageref{tvbc});  
 \label{kfbew}     $\mathbf{f}_c$ is a holomorphic correspondence and $ R>0$  is an escaping radius (page \pageref{podemg}); 
   $z \mapsto z_1 \mapsto \cdots$ is a forward orbit (page \pageref{pwerf});

 \item    $\mathbf{f}_c^{-1}(w) $  and     $\mathbf{f}_c(z)$ are defined by  $\{z: (w-c)^q=z^p\}$ and $\{w: (w-c)^q=z^p\}$ (page \pageref{pdmge}); 
$\mathbf{f}_c(A)$ and  $\mathbf{f}_c^{-1}(A)$ are defined by $\bigcup_{z\in A} \mathbf{f}_c(z)$  and  $\bigcup_{z\in A} \mathbf{f}_c^{-1}(z)$ (page \pageref{phgen});
$A \Subset B$  means that  $A\subset E\subset B$ for some compact set $E$ (page \pageref{gdgihg});

  \item $M_{\beta}$ is the  connectedness locus and \  $M_{\beta,0}=\{c \in \mathbb{C}: 0 \in K_c\}$ (page \pageref{qwegn}); every open ball centred at zero is denoted by  $B_r=\{z\in \mathbb{C}: \ |z| < r\}$ (page \pageref{ugnew});
 $\mathbf{g}_c$ is defined by   $\mathbf{g}_c(z)=\mathbf{f}_c(z)\cap K_c,$ where $K_c$ is the filled Julia set  of $\mathbf{f}_c$ (page  \pageref{gueng});
   $\varphi$ is a  branch of $\mathbf{f}_c$ or $\mathbf{f}_c^{-1}$ (page \pageref{pjyfk});

     \item   $\mathcal{S}_{c,d}$ is the sequence of maps $B_d \xrightarrow{\mathbf{f}_c} D_{1,c} \xrightarrow{\varphi_{1,c}}  \cdots \xrightarrow{\varphi_{n-1,c}} D_{n,c}$  (page
   \pageref{poihyt}); 
   $J_c$ and  $J_c^*$ are the Julia set and dual Julia set of $\mathbf{f}_c$ (page   \pageref{qegyc});
      $\operatorname{dom}(\mathcal{F}_c)$  is the  domain of a non-critical system (page \pageref{pitgk});

   \item    $\operatorname{dom}(\mathcal{A}_c)$  is the domain of a critical system (page \pageref{sdgrw});
   $\mathcal{F}_{c}$  is a non-critical system and   $\mathcal{A}_{c}$  is a critical system (page \pageref{sdgrw});
   $\omega(z, \mathbf{g}_c)$ is the $\omega$-limit set under $\mathbf{g}_c$ (page   \pageref{pogew});

    \item   $\Lambda(\mathcal{F}_c)$  is the limit set of a non-critical system (page  \pageref{rcv});
$P_c$  is the post-critical set of $\mathbf{f}_c$ (page  \pageref{qetdbc});
$H_{\beta}$ is  a subset consisting of hyperbolic parameters (page  \pageref{sdboidg});

    \item     $d_c$ is the  hyperbolic distance function of $\hat{\mathbb{C}} - P_c$ (page \pageref{lkghegd});
  $\alpha(z, \mathbf{f}_c)$ is the  $\alpha$-limit set of $z$ (page \pageref{wcmg});
$\mathbf{h}_c$ denotes a   branched holomorphic motion (page  \pageref{gehgds});
$\{ |\mathbf{f}_c(z)| >r\}$ is a short version of the set  $\{z\in \mathbb{C}: |w| >r  \ \textrm{for any}  \  w\in \mathbf{f}_c(z)\}$ (page    \pageref{huiji});

    \end{enumerate}

 \section*{Acknowledgments.}  Research  partially supported by the  grants 2016/16012-6 S\~ao Paulo Research Foundation, and CNPq 232706/2014-0. The author would like to thank Daniel Smania for many discussions concerning the central problems of this article, and Shaun Bullett, Sylvain Bonnot,  Christopher Penrose, and Luna Lomonaco for the valuable comments and suggestions.

   This paper has been written during  the author's postdoctoral stage at  Imperial College London and  University of S\~ao Paulo:  we are very grateful to Sebastian van Strien and Edson de Faria (resp.).

Approximately one-fourth of this paper originated from \cite{Carlos}; some other parts have been announced (without proofs) in \cite{BLS}.

\bibliographystyle{abbrv}

{\small
\bibliography{oi}  }

\vspace{1cm}

  C.~Siqueira, \textsc{Department of Mathematics, Federal University of Bahia, Salvador, Bahia, 
    Brazil, CEP  40170-115}\par\nopagebreak
  \textit{E-mail address}, C.~Siqueira: \texttt{carlos.siqueira@ufba.br}

\end{document}